\documentclass{article}
\usepackage[utf8]{inputenc}
\usepackage{amsmath}
\usepackage{amssymb}
\usepackage{amsthm}
\usepackage{authblk}
\usepackage{enumitem}
\usepackage{thm-restate}
\usepackage{hyperref}
\usepackage{color,soul}

\newtheorem{theorem}{Theorem}
\newtheorem{conjecture}[theorem]{Conjecture}
\newtheorem{corollary}[theorem]{Corollary}
\newtheorem{lemma}[theorem]{Lemma}
\newtheorem{proposition}[theorem]{Proposition}
\newtheorem*{theoremL(1)}{Theorem~\ref{structure of L(1)}}
\newtheorem*{theoremL(n)}{Theorem~\ref{L(n) for large n}}
\newtheorem*{conjectureL(2)}{Conjecture~\ref{structure of L(2)}}

\theoremstyle{definition}
\newtheorem*{definition*}{Definition}
\newtheorem*{example*}{Example}
\newtheorem*{notation*}{Notation}
\newtheorem*{remark*}{Remark}
\newtheorem{remark}[theorem]{Remark}

\newcommand{\N}{\mathbb{N}}

\newcommand{\seq}[1]{\cite[\href{http://oeis.org/#1}{#1}]{OEIS}}

\title{The lexicographically least square-free word \\ with a given prefix\thanks{This work was done in the 2021 Polymath Jr.\ program and was partially supported by NSF award DMS-2113535.}}

\author[1]{Siddharth~Berera}
\affil[1]{University of Edinburgh}
\author[2]{Andr\'es~G\'omez-Colunga}
\affil[2]{Pennsylvania State University}
\author[3]{Joey~Lakerdas-Gayle}
\affil[3]{University of Waterloo}
\author[4]{John~L\'opez}
\affil[4]{Tulane University}
\author[5]{Mauditra~Matin}
\affil[5]{Delft University of Technology}
\author[6]{Daniel~Roebuck}
\affil[6]{University of St Andrews}
\author[7]{Eric~Rowland}
\affil[7]{Hofstra University}
\author[8]{Noam~Scully}
\affil[8]{Yale University}
\author[9]{Juliet~Whidden}
\affil[9]{Vassar College}

\date{\today}

\begin{document}

\maketitle

\begin{abstract}
The lexicographically least square-free infinite word on the alphabet of non-negative integers with a given prefix $p$ is denoted $L(p)$.
When $p$ is the empty word, this word was shown by Guay-Paquet and Shallit to be the ruler sequence.
For other prefixes, the structure is significantly more complicated.
In this paper, we show that $L(p)$ reflects the structure of the ruler sequence for several words $p$. We provide morphisms that generate $L(n)$ for letters $n=1$ and $n\geq3$, and $L(p)$ for most families of two-letter words $p$.
\end{abstract}

%%%%%%%%%%%%%%%%%%%%%%%%%%%%%%%%
\section{Introduction}

A word is \emph{square-free} if it contains no block of letters that occurs twice consecutively.
In 2009, Guay-Paquet and Shallit~\cite{Guay-Paquet--Shallit} established the structure of the lexicographically least square-free infinite word on the alphabet $\N:=\{0, 1, 2,\dots\}$. This word is $0102010301020104\cdots$.
Its letters comprise the ruler sequence~\seq{A007814}, and it is the fixed point $\rho^\infty(0)$ of the \emph{ruler morphism} $\rho$ defined by $\rho(n) = 0 \, (n + 1)$.

However, this result is not robust; a minor variation produces words that are quite different.
Given a word $w$, let $L(w)$ denote the lexicographically least infinite word on $\N$ beginning with $w$ whose only square factors are contained in the prefix $w$. In particular, if $w$ is square-free, then so is $L(w)$. If $w$ is infinite, then $L(w)=w$. For example,
\[
    L(1) = 10120102012021012010201203010201\cdots\quad\textrm{\seq{A356677}},
\]
\[
    L(2) = 20102012021012010201202102010210\cdots\quad\textrm{\seq{A356678}},
\]
and 
\[
    L(33) = 33010201030102012021012010201202\cdots\quad\textrm{\seq{A356679}}.
\]
Unlike $L(\varepsilon) = \rho^\infty(0)$, the letters in these words do not alternate between $0$s and positive integers.
Moreover, the letters $3, 4, 5, \dots$ take much longer to appear, and there is no clear pattern for these words, as in the case for the ruler sequence.

In this paper we determine the structure of $L(w)$ for certain simple words $w$.
Our main results are that $L(1)$ and $L(n)$ for $n\ge 3$ reflect the structure of $L(\varepsilon)$ as follows.

\begin{theorem}\label{structure of L(1)}
There exists a morphism $\alpha$ and a word $Y_1$ with length $5177$ such that $L(1) = Y_1 \, \alpha(L(\varepsilon))$.
\end{theorem}

\begin{theorem}\label{L(n) for large n}
For each $n \geq 3$, there exists a finite word $Y_n$ such that $L(n) = Y_n\, \rho(\alpha(L(\varepsilon)))$, where $\alpha$ is the morphism in Theorem~\ref{structure of L(1)}.
\end{theorem}

These results imply that the suffix $\alpha(L(\varepsilon))$ and the related suffix $\rho(\alpha(L(\varepsilon)))$ are two attractors for infinite square-free extensions of words on $\N$. We have
\[
    \alpha(L(\varepsilon)) = 01020301201020120210120102012023\cdots\quad\textrm{\seq{A356676}}.
\]

The structure of $L(2)$ appears to also reflect the structure of $L(\varepsilon)$ but, surprisingly, appears not to be related to the morphism $\alpha$.
Instead, we give a morphism $\gamma$ in Section~\ref{L(2) section} for which we conjecture the following.

\begin{conjecture}\label{structure of L(2)}
There exists a morphism $\gamma$ such that $L(2) = 2 \gamma(L(\varepsilon))$.
\end{conjecture}

In Corollary~\ref{L(012)}, we show that $L(012)=01201\lim_{n\to\infty}\rho^{-1}(\alpha(n))$ where $\alpha$ is the morphism described above. We also describe $L(p)$ for many two-letter words. For example, we show that $L(nn)=nL(n)$ for all letters $n$. On the other hand, we do not have conjectures for the structures of $L(1n)$ when $n>1$ and $L(2n)$ when $n\not\in\{0,2\}$. For example we have
\begin{align*}
    L(12) &= 12010201202101201020120212010201\cdots\quad\textrm{\seq{A356680}},\\
    L(13) &= 13010201030102012021012010201202\cdots\quad\textrm{\seq{A356681}},\\
    L(21) &= 21012010201202101201020121012010\cdots\quad\textrm{\seq{A356682}},\\
    L(23) &= 23010201030102012021012010201202\cdots\quad\textrm{\seq{A356683}}.
\end{align*}

Section~\ref{preliminary results} contains the main definitions and some preliminary results we will use to prove Theorems~\ref{structure of L(1)} and \ref{L(n) for large n}.
Section~\ref{section P prefixes} establishes the structure of finite words $T(n)$ for which $Y_n=nT(n)A$ for a constant word $A$ when $n\geq3$. We see in Section~\ref{L(1) section} that $T(n)$ is the only component of $L(n)$ that depends on $n$. In Section~\ref{section: one-letter prefixes} we give explicit constructions of $\alpha$ and the words $Y_n$, and we prove Theorems~\ref{structure of L(1)} and \ref{L(n) for large n}. Unsurprisingly, the proofs are fairly technical.

In Section~\ref{section: extending} we give two conditions under which $L(u v) = u \, L(v)$.
For example, we show that $L(n_1 n_2) = n_1 L(n_2)$ for all $n_1 \geq 3$ and $n_2 \geq 3$. We use this to describe the structure of $L(p)$ for several families of two-letter words $p$.
Finally, in Section~\ref{section: inducing factors}, we study the inverse problem of finding a prefix $p$ that induces a given finite square-free word $w$ under $L$.
We show that this can always be done and present an algorithm that computes a prefix $p$ such that $p w$ is a prefix of $L(p)$. Section~\ref{glossary} is a glossary of all of the functions, morphisms, and constants defined in the paper.

This work was motivated by results of several papers~\cite{Rowland--Shallit, Pudwell--Rowland, Rowland--Stipulanti} studying the lexicographically least word on $\N$ that avoids $\frac{a}{b}$-powers, for various rational numbers in the interval $1 < \frac{a}{b} < 2$.
These words exhibit a remarkable diversity of behaviors.
Some of these words (for example, when $\frac{a}{b} = \frac{24}{17}$~\cite[Figure~4 on page~36]{Pudwell--Rowland}) alternate between two different modes before settling into their long-term behavior, suggesting that there may be multiple attractors in the dynamical systems that generate them. The current paper is the first exploration of the set of attractors for the alphabet $\N$.
By varying the prefix, we show there are multiple attractors for square-free words.
In contrast, on a binary alphabet there are patterns with only one attractor; Allouche, Currie, and Shallit~\cite{Allouche--Currie--Shallit} showed that the lexicographically least overlap-free word on $\{0, 1\}$ with a fixed prefix, if it exists, always has a suffix that is a suffix of $10010110\cdots$ (the complement of the Thue--Morse word).

Frequently in this paper, we use computations to show that a particular finite word possesses a certain property. Often, this involves verifying that the word is square-free. As an example, in the proof of Lemma~\ref{psi_1(0n0) square-free}, we verify that the word $\psi_1(010)$ is square-free. These words can easily be computed from their definitions. All computations in this paper that verify that a word is square-free can be done quickly using the Main--Lorentz algorithm described in their 1985 paper \cite{Main--Lorentz}.

We sometimes require the computation of finite prefixes of words $L(w)$. For example, the word $Y_1$ is defined as the $5177$-letter prefix of $L(1)$. These computations can be completed quickly using the following simple greedy algorithm: To find the next letter of $L(w)$ after the prefix $wv$, check whether $wvn$ has a square suffix for non-negative integers $n$, increasing from zero. Then the least $n$ for which $wvn$ has no square suffix is the next letter of $L(w)$. If $w$ contains a square, we test all even-length suffixes of $wvn$. If $w$ is square-free, then we can use a more efficient variation of the Main--Lorentz algorithm~\cite{Main--Lorentz} that only searches for squares that are suffixes. This variation is easiest to implement by first reversing the string and then using the algorithm described in their paper with ``Step 3" altered so that only the first block of length $l$ is examined.
%%%%%%%%%%%%%%%%%%%%%%%%%%%%%%%%
\section{Definitions and preliminary results}\label{preliminary results}

We assume the reader is familiar with basic definitions and notation regarding words and morphisms. See the survey by Allouche~\cite{Allouche} for a short introduction.
All words in the paper are words on the alphabet $\N$.

We say that $v$ is a \emph{factor} of $w$ if $w = x v y$ for some words $x, y$.
If $v$ is a factor of $w$, we also say that \emph{$v$ occurs in $w$} and \emph{$w$ contains $v$}.
A \emph{square} is a nonempty word of the form $yy$.
A word is \emph{square-free} if it contains no square factors.

We index letters in a word beginning with $0$, and we use Python notation to extract factors: Suppose a word $w$ has length $\ell$ and is written as a sequence of letters $w = w_0 w_1 w_2 \cdots w_{\ell-2} w_{\ell-1}$. If $0 \leq i \leq j \leq \ell - 1$, then $w[i] := w_i$ and $w[i:j] := w_i w_{i+1} \cdots w_{j-1}$, with ``default values'' for $i$ and $j$ being 0 and $\ell$, respectively. If $i=j$, then $w[i:j]=\varepsilon$. Negative values index letters from the end of the word.
For example, if $w = \textsf{hydrant}$, then
\begin{align*}
        w[3:6] &= \textsf{ran} = w[3:-1] \\
        w[4:] &= \textsf{ant} = w[-3:] \\
        w[:5] &= \textsf{hydra} = w[:-2].
\end{align*}
Note also that $w[:0]=\varepsilon$ for any word $w$.

\begin{definition*}
A word $w$ is \emph{even-grounded} if $w_i = 0$ for even $i$ and $w_i \neq 0$ for odd $i$. A word $w$ is \emph{odd-grounded} if $w_i = 0$ for odd $i$ and $w_i \neq 0$ for even $i$. A word is \emph{grounded} if it is even-grounded or odd-grounded.
\end{definition*}

For example, the words $010$ and $0102$ are even-grounded, $301$ and $3010$ are odd-grounded, and the words $0120$ and $0100$ are not grounded.

\begin{definition*}
The \emph{ruler morphism} $\rho\colon \N^* \cup \N^\omega \to \N^* \cup \N^\omega$ is defined by $\rho(n) = 0(n+1)$ for letters $n \in \N$.
\end{definition*}

\begin{notation*}
We will denote by $R_{n}$ the prefix of the ruler sequence up to the first occurrence of the letter $n$, i.e.\ $R_n=\rho^n(0)=L(\varepsilon)[:2^{n}]$.
For example, $R_0 = 0$, $R_1 = 01$, $R_2 = 0102$, and so on.
\end{notation*}

\begin{notation*}
For a nonempty finite word $w$, we define $w^+$\!, the \emph{successor} of $w$, to be the word that is identical to $w$ except for the last letter, which is increased by $1$. For example, if $w = 0102$ then $w^+ = 0103$. Formally, $w^+ = w[:-1](w[-1]+1)$.
\end{notation*}

\begin{definition*}
Let $\phi$ be a morphism.
\begin{itemize}
    \item $\phi$ is \emph{non-erasing} if $\phi(k) \neq \varepsilon$ for all letters $k$.
    \item Let $\Delta$ be a set of words. We say that $\phi$ is \emph{square-free over $\Delta$} if $\phi$ is non-erasing and $\phi(w)$ is square-free for all square-free words $w \in \Delta$. We say that $\phi$ is \emph{square-free} if $\phi$ is square-free over $\N^*\cup\N^\omega$.
    \item $\phi$ is \emph{letter-injective} if, given letters $k$ and $\ell$, $\phi(k) = \phi(\ell)$ implies $k = \ell$.
\end{itemize}
\end{definition*}

\begin{definition*}
Given two words $u$ and $v$, we say that $u$ is \emph{lexicographically less than $v$} and we write $u\prec v$ if there is an index $i$ such that $u[:i]=v[:i]$ and $u[i]<v[i]$ as letters. It can be seen that $\prec$ is a partial ordering on $\N^*\cup\N^\omega$. The only case when $u$ and $v$ are not comparable by $\prec$ is when one word is a prefix of the other. We can also see that if $w$ is a nonempty word of finite length and $v\prec w^+$, then either $v\prec w$ or $w$ is a prefix of $v$.
\end{definition*}

\begin{definition*}
Given words $w$ and $u$, we say that $u$ is \emph{irreducible in $wu$} if for all words $s\prec u$, $s$ introduces a square in $ws$. That is, there is a square in $ws$ that ends at a letter in $s$.
\end{definition*}
Notice that if $u$ is irreducible in $wu$, it can still introduce a square. In this case, $u^+$ is irreducible in $wu^+$.

\begin{example*}
Consider $w=0102010$, $u=23$, and $v=301$. Any word $y$ beginning with a 0, 1, or 2 introduces a square in $wy$, so $u$ is irreducible and also introduces a square in $wu$. Thus, $u^+=24$ is also irreducible. The word $v=301$ does not introduce a square, but every word lexicographically less than $v$ either begins with 0, 1, or 2, or includes a square. So $v$ is irreducible, but $v^+=302$ is not.
\end{example*}

\begin{definition*}
We say that \emph{$p$ generates $ps$} if $L(p)=L(ps)$. In other words, $L(p)$ starts with $ps$.
\end{definition*}

The main results of this paper, that $L(1)=Y_1\alpha(L(\varepsilon))$ and $L(n)=Y_n\rho(\alpha(L(\varepsilon)))$ for $n\geq3$ can be restated as 1 generates $Y_1\alpha(L(\varepsilon))$, and $n$ generates $Y_n\rho(\alpha(L(\varepsilon)))$ for $n\geq3$.

\begin{remark}\label{Appending to front of lex-least words}
To prove that a square-free word $p$ generates another word $w=ps$, we must show that $w$ is square-free and that $s$ is irreducible in $w=ps$.

In particular, we have the useful property that if $p$ generates $w=ps$ and $uw$ is square-free for another word $u$, then $up$ generates $uw$. The square-free condition is known by assumption, and the fact that $s$ is irreducible in $w=ps$ implies that $s$ is irreducible in $uw=ups$.

One common application of this property is when $w=L(p)$. Then if $uL(p)$ is square-free, we get $L(up)=L(uL(p))=uL(p)$. Or in other words, $up$ generates $uL(p)$.
\end{remark}

\begin{example*}
The above remark is often used implicitly in this paper. For example in Lemma~\ref{psi_1(n) generates psi_1(n0)}, we prove that for $n\geq1$, $\psi_1(n)$ generates $\psi_1(n)202101$. To do this, we first show that $\psi_1(n)202101$ is square-free and that 2 is irreducible in $\psi_1(n)2$. This proves that $\psi_1(n)$ generates $\psi_1(n)2$.

We then show that $\psi_1(n)2$ has suffix $R_2R_12$, and via computation show that $R_2R_12$ generates $R_2R_1202101$. Letting $p=R_2R_12$, $s=02101$, and $u$ be the word for which $up=\psi_1(n)2$, we have that $p$ generates $ps$, and $ups=\psi_1(n)202101$ is square-free. So Remark~\ref{Appending to front of lex-least words} says that $up=\psi_1(n)2$ generates $ups=\psi_1(n)202101$.

We now have that $\psi_1(n)$ generates $\psi_1(n)2$, and $\psi_1(n)2$ generates $\psi_1(n)202101$. By definition, this means that $L(\psi_1(n))=L(\psi_1(n)2)=L(\psi_1(n)202101)$, so $\psi_1(n)$ generates $\psi_1(n)202101$.
\end{example*}

\begin{notation*}
For any word $w$, $\max(w)$ denotes the maximum letter value in $w$ if it exists.
\end{notation*}

The remainder of this section introduces the notion of \emph{chunks} which will be key in Section~\ref{section P prefixes} to prove Theorem~\ref{tau prefix theorem} regarding the structure of a certain prefix of $L(n)$ and in Section~\ref{section: one-letter prefixes} to prove Theorem~\ref{alpha square-free over grounded words} related to the square-freeness of the morphism $\alpha$, determining the structure of $L(1)$.

Given a morphism $\phi$ and a word $w=w_0w_1w_2\cdots $ we can write $\phi(w)$ as
\[
	\phi(w)=[\phi(w_0)] \ [\phi(w_1)] \ [\phi(w_2)] \cdots,
\]
where we use square brackets to delineate the contributions of each individual letter of $w$. A factor of $\phi(w)$ that arises as the image $\phi(k)$ of the letter $k$ under $\phi$ is a \emph{chunk} or, more specifically, a \emph{$k$-chunk}.
For example, consider the ruler morphism $\rho$ and the word $0121$. We would break $\rho(0121)$ into chunks as
\[
	\rho(0121) = [01] \ [02] \ [03] \ [02].
\]
Each factor $02$ is a $1$-chunk.

Sometimes, when we take a particular occurrence of a factor of $\phi(w)$, we find that the factor starts or ends partway through a chunk. For example $\rho(0121)[1:6]$ can be written as
\[
	\rho(0121)[1:6] = 1] \ [02] \ [03] \ [0.
\]
We refer to $1]$ and $[0$ as \emph{partial chunks}. $1]$ is the \emph{initial} partial chunk and $[0$ is the \emph{final} partial chunk.

In this paper, we frequently consider words of the form $\phi(w)$, and then characterize the possible locations of certain factors of $\phi(w)$ with respect to its chunks. For example, in Lemma~\ref{E is prefix or suffix of alpha(0)}, we show that for any grounded square-free word $w$, the constant word $E$ can only occur in $\alpha(w)$ as a prefix or a suffix of a 0-chunk. This is related to the morphism property that \emph{$\phi$ locates words of length $\ell$} introduced by Pudwell and Rowland~\cite{Pudwell--Rowland}, which restricts the possible starting index for a word of length $\ell$, relative to the chunks in $\phi(w)$.
It is also related to the \emph{synchronization point} of a word $x$ introduced by Cassaigne~\cite{Cassaigne} which describes a point in a word $x$ with respect to a morphism $\phi$ that must occur at a chunk boundary whenever $x$ occurs in $\phi(w)$.

\begin{definition*}
Given a word $w$ and a morphism $\phi$, we say that two (possibly partial) chunks in $\phi(w)[i:j]$ \emph{come from the same letter} to mean that both chunks arise as images of the same letter in $w$. For example, in $\rho(0121)[1:6]=1][02][03][0$, the first whole chunk, $[02]$, and the final partial chunk, $[0$, come from the same letter, 1.
If a word $v$ occurs twice in $\phi(w)$, we say that the occurrences have the same \emph{chunk decomposition} if every whole chunk in one occurrence corresponds to a whole chunk in the other occurrence and both chunks come from the same letter.
\end{definition*}

\begin{example*}
Consider the morphism $\phi\colon \N^* \cup \N^\omega \to \N^* \cup \N^\omega$ defined by
\[\phi(n)=\begin{cases}
	0 & \text{if $n=0$} \\
	01 & \text{if $n=1$}\\
	n-1 & \text{if $n\ge 2$}.
\end{cases}
\]
The single occurrence of $v=103$ in $\phi(2041) =[1] \ [0] \ [3]\ 01$ has the same chunk decomposition as its occurrence in $\phi(2204)=1 \ [1] \ [0]\ [3]$. On the other hand, $w=01$ in $\phi(01) =0 \ [01]$ does not have the same chunk decomposition as its occurrence in $\phi(024)= [0] \ [1]\ 3$.

In $\rho(01012)$, we consider the square $1020\ 1020$:
\[\rho(01012)=[0\ (1][02][0)\ (1][02][0)\ 3].\]
The two halves of the square have the same chunk decomposition because in both, $02$ is a whole 1-chunk and there are no other whole chunks. The two initial partial chunks come from the same letter because they are both part of 0-chunks. The two final partial chunks do not come from the same letter because the final partial chunk of the first half is part of a 0-chunk, and the final partial chunk of the second half is part of a 2-chunk.

Consider a morphism $\phi$ where $\phi(0)=123$, and $\phi(1)=13$. Then, $\phi(010)=[12(3][1)(3][1)23]$. In the square $3131$, both halves vacuously have the same chunk decomposition since there are no whole chunks. But neither their initial nor final partial chunks come from the same letter.
\end{example*}

The following theorem is used frequently in this paper to show that a morphism is square-free.

\begin{theorem}\label{chunk-identifiable condition for square-freeness} Let $\phi$ be a letter-injective morphism such that $\phi(\ell)$ is square-free for all letters $\ell$. Suppose $w$ is a word such that $\phi(w)$ contains a square $yy$ for which the following three properties hold.
\begin{enumerate}
	\item Each half of the square contains at least one whole chunk.
	\item The two halves of the square have the same chunk decomposition.
	\item If either half has a partial chunk, then either their initial partial chunks or their final partial chunks come from the same letter in $w$.
\end{enumerate}
Then $w$ contains a square.
\end{theorem}

\begin{proof}
By the first and second conditions, the corresponding whole chunks in both halves come from the same letters in $w$. So there is a nonempty factor $u$ of $w$ that yields all whole chunks in both halves. That is, we can write the square as
\[yy=(b]\ [\phi(u)]\ [c)\ (b]\ [\phi(u)]\ [c)\]
where $b]$ and $[c$ are possibly empty partial chunks.

If there are no partial chunks ($b=c=\varepsilon$), then $yy=[\phi(u)][\phi(u)]$ and $uu$ is a square factor of $w$.

Now, suppose there is a partial chunk in either half. Then since the halves have the same chunk decomposition, there must be an initial and final partial chunk in both halves (neither $b$ nor $c$ is empty). Let $a$ and $d$ be the smallest words that complete the partial chunks of $yy$. That is, $\phi(w)$ has the factor $[a(b]\ [\phi(u)]\ [c)(b]\ [\phi(u)]\ [c)d]=ayyd$. Since $\phi$ is letter-injective, $\phi^{-1}$ is well defined on chunks. The third condition says that either $\phi^{-1}(ab)=\phi^{-1}(cb)$ or $\phi^{-1}(cb)=\phi^{-1}(cd)$. In the first case, $\phi^{-1}(ab)u\phi^{-1}(cb)u$ is a square factor of $w$. In the second case, $u\phi^{-1}(cb)u\phi^{-1}(cd)$ is a square factor of $w$.
\end{proof}

As a corollary to this theorem, suppose $\Gamma$ is a family of square-free words, and $\phi$ is a letter-injective morphism with $\phi(\ell)$ square-free for all letters $\ell$. For any word $w\in\Gamma$, $\phi(w)$ cannot contain a square that possesses all three properties in Theorem~\ref{chunk-identifiable condition for square-freeness}, or else $w$ would contain a square. So we can prove that $\phi$ is square-free over $\Gamma$ by showing that for all $w\in\Gamma$, if $\phi(w)$ contains a square, then it contains a square with these three properties. In this paper, we use this technique to show that the morphisms $\psi_1$, $\psi_2$, and $\alpha$ are square-free over grounded words.

%%%%%%%%%%%%%%%%%%%%%%%%%%%%%%%%
\section{A certain prefix of $L(n)$} \label{section P prefixes}

For all $n\geq1$, it is clear that $L(n)$ always starts with the letter $n$ followed by a prefix of the ruler sequence $L(\varepsilon)$. In this section, we show that for $n\geq3$, $L(n)$ has prefix $nT(n)$ which has length exponential in $n$. In Section~\ref{L(1) section} we will show that this is the only part of $L(n)$ that depends on $n$.

\subsection{The morphism $\psi_1$}
We begin by showing that $L(n)$ has a shorter prefix, $nP_0(n)P_1(n)$, which is proved in Theorem~\ref{P1 theorem}. Next we define the words $P_0(n)$, $P_1(n)$ and the morphism $\psi_1$.

\begin{definition*}
For $n\geq 0$, let $P_0(n)$ be the maximum prefix of the ruler sequence such that $nP_0(n)$ is a prefix of $L(n)$. Define the morphism $\psi_1 \colon \N^\ast\to\N^\ast$ by
\[\psi_1(n)=\begin{cases}
	202101 & \text{if $n=0$},\\
	(n+1) \, P_0(n+1) & \text{otherwise},
\end{cases}
\]and for $n\geq3$ define
\[P_1(n)=\psi_1(P_0(n-1)).\]
\end{definition*}

For $n\geq3$, after the word $L(n)$ deviates from the ruler sequence prefix, we will show that it continues with the word $P_1(n)$. These definitions can be referenced in the glossary, Section~\ref{glossary}.

\begin{remark}\label{structure of P0}
Note that $P_0(0)=\varepsilon$. Also, it is not hard to see that the length of $P_0(n)$ is $2^{n+1}-2$, hence for all $n$ we have $$P_0(n)=R_{n+1}[:-2]=R_nR_n[:-2]=R_nR_{n-1}\cdots R_3R_2R_1.$$ And for $n\geq 1$, $$P_0(n)=R_nP_0(n-1).$$ By repeated application of this argument, for $1\le k \le n$, we have that $P_0(n)$ has suffix $P_0(k)$.
\end{remark}

We would like to show that $\psi_1$ is square-free over grounded words, since that will imply that $P_1(n)$ is square-free. The next three lemmas describe some of the important behavior of $\psi_1$ over grounded square-free words.

\begin{lemma}\label{psi_1(0n0) square-free}
Let $n\geq1$. Then $\psi_1(0n0)$ is square-free.
\end{lemma}
\begin{proof}
We can verify computationally that $\psi_1(010)$ is square-free. So let $n\geq2$ and suppose there is a square $yy$ in
\[\psi_1(0n0)=\big[202101\big]\big[(n+1)P_0(n+1)\big]\big[202101\big].\]
Since $(n+1)P_0(n+1)$ is a prefix of $L(n+1)$, $\psi_1(n)$ is square-free for all $n$. So $yy$ lies over at least one of the chunk boundaries.

In the first chunk boundary we find the word $1(n+1)$, and in the second the word $12$. Since $n+1\geq3$ and $\psi_1(n)$ is grounded, these words appear nowhere else in $\psi_1(0n0)$. Therefore, neither chunk boundary can be completely contained in $y$. This means that $yy$ must be a factor of $\psi_1(0n)$ or $\psi_1(n0)$ and each occurrence of $y$ must be completely contained in a different chunk. Every prefix of $\psi_1(n)$ begins with $n+1\geq3$ but no suffix of $\psi_1(0)$ does, and every suffix of $\psi_1(n)$ is grounded and ends with a 1 which is not true for any prefix of $\psi_1(0)$. Therefore, such a square cannot exist.
\end{proof}

\begin{lemma}\label{psi_1(n) not factor of psi_1(k)}
Let $n>k\geq0$. Then neither of $\psi_1(n)$ and $\psi_1(k)$ is a factor of the other.
\end{lemma}
\begin{proof}
From the definition, it is clear that since $n>k$, $|\psi_1(n)|>|\psi_1(k)|$, so $\psi_1(n)$ cannot be a factor of $\psi_1(k)$. If $k=0$, then $\psi_1(k)$ is not grounded but $\psi_1(n)$ is. Thus, assume $n>k>0$. Then since $\psi_1(n)=(n+1)P_0(n+1)=(n+1)R_{n+2}[:-2]$, it is sufficient to show that $\psi_1(k)=(k+1)R_{k+2}[:-2]$ cannot occur in $R_{n+2}$.

Since $R_{n+1}=R_nR_n^+$ for all $n$, we have that
\begin{align*}
    R_{n+2} &= R_{k+2}R_{k+2}^+R_{k+2}R_{k+2}^{++}R_{k+2}R_{k+2}^{+}\cdots R_{k+2}R_{k+2}^{++\cdots +}\\
    & =R_{k+2}^*R_{k+2}^*\cdots R_{k+2}^{*},
\end{align*}
where $^*$ represents the application of zero or more $^+$'s according to the ruler sequence pattern. 

The largest letter in $R_{k+2}[:-2]$ is $k+1$ and the last letter of each $R_{k+2}^*$ is at least $k+2$. So every occurrence of $R_{k+2}[:-2]$ in $R_{k+2}^*R_{k+2}^*\cdots R_{k+2}^*$ must be as a prefix of an $R_{k+2}^*$. This means that any occurrence of $R_{k+2}[:-2]$ in $R_{n+2}$ is either a prefix of $R_{n+2}$ or is preceded by a letter that is at least $k+2$. Therefore, $(k+1)R_{k+2}[:-2]$ can never occur in $R_{n+2}$, and $\psi_1(k)$ cannot occur in $\psi_1(n)$ which proves the result.
\end{proof}

\begin{lemma}\label{psi_1(l) is a l-chunk}
Let $w$ be a grounded square-free word. Then for $\ell\geq0$, every occurrence of $\psi_1(\ell)$ in $\psi_1(w)$ is a $\ell$-chunk.
\end{lemma}
\begin{proof}
It is not hard to see that the word $21$ can only occur in $\psi_1(w)$ in the middle of 0-chunks. Thus, $202101$ occurs in $\psi_1(w)$ only as a 0-chunk.

Suppose $\ell\geq1$. Then $\psi_1(\ell)=(\ell+1)P_0(\ell+1)=(\ell+1)R_{\ell+1}R_{\ell+1}[:-2]$ which is odd-grounded. For $n\geq0$, $\psi_1(n)$ begins and ends with a non-zero letter, so any word lying over a chunk boundary in $\psi_1(w)$ cannot be grounded. Therefore, every occurrence of $\psi_1(\ell)$ in $\psi_1(w)$ must be totally contained in some $k$-chunk where $k\geq\ell$. By Lemma~\ref{psi_1(n) not factor of psi_1(k)} $\psi_1(\ell)$ cannot be a factor of $\psi_1(k)$ when $k>\ell$, so $\psi_1(\ell)$ can only occur as an $\ell$-chunk.
\end{proof}

\begin{proposition}\label{psi_1 square-free over grounded}
$\psi_1$ is square-free over grounded words.
\end{proposition}
\begin{proof}
Suppose $w$ is a grounded square-free word and that $\psi_1(w)$ contains a square $yy$. From its definition, $\psi_1$ is letter injective and $\psi_1(n)$ is square-free for all $n$. Also, Lemma~\ref{psi_1(l) is a l-chunk} implies that both halves of $yy$ have the same chunk decomposition. We will show that each half of the square contains a whole chunk, and that if either half contains a partial chunk, then the final partial chunk of both halves comes from the same letter. Then Theorem~\ref{chunk-identifiable condition for square-freeness} will imply that $w$ contains a square, which is a contradiction.

% Both halves contain a whole chunk
Suppose neither half of the square contains a whole chunk. Then the whole square contains no more than one whole chunk. By Lemma~\ref{psi_1(0n0) square-free}, $\psi_1(0n0)$ is square-free, so $yy$ must be a proper factor of $\psi_1(n0k)$ for $n,k>0$, $n\neq k$. The factor $21$ only occurs in the middle of 0-chunks, so the square in $\psi_1(n0k)$ has its center at the $21$ in the 0-chunk. The square cannot be totally contained in the 0-chunk, so the second half of the square begins with $101$ which cannot occur in the first half because $\psi_1(n)=(n+1)R_{n+2}[:-2]$. This is a contradiction so at least one half of the square contains a whole chunk. Since both halves have the same chunk decomposition, both halves contain a whole chunk.

Suppose either half contains a partial chunk. Then since the halves have the same chunk decomposition, they must both end with a partial chunk. The final partial chunks begin with the same letter, so they must be equal chunks or one of them must be a 0-chunk. But since the halves have the same chunk decomposition and contain whole chunks, their last whole chunk is equal. So the final partial chunks are either both 0-chunks or both equal nonzero chunks.
\end{proof}

The next three lemmas are used to prove the irreducibility condition in Theorem~\ref{P1 theorem}.

\begin{lemma}\label{psi_1(n) generates psi_1(n0)}
Let $n\geq1$. Then $\psi_1(n)$ generates $\psi_1(n0)$.
\end{lemma}
\begin{proof}
Since $n0$ is square-free and grounded, $\psi_1(n0)$ is square-free by Proposition~\ref{psi_1 square-free over grounded}, so we only need to show that $\psi_1(0)$ is irreducible in $\psi_1(n0)=\psi_1(n)202101$. We will first show that 2 is an irreducible suffix of $\psi_1(n)2$, and then that $02101$ is an irreducible suffix of $\psi_1(n)202101$.

Recall that $\psi_1(n)=(n+1)P_0(n+1)=(n+1)R_{n+1}[:-1](n+1)R_{n+1}[:-2]$. From the structure of $R_n$, we know that $R_{n+1}$ has suffix $010(n+1)$ when $n\geq1$. So $\psi_1(n)0=(n+1)R_{n+1}[:-1](n+1)R_{n+1}[:-1]$ which is a square. The last letter of $\psi_1(n)$ is a 1, so $\psi_1(n)1$ has square suffix $11$. Therefore, 2 is irreducible at the end of $\psi_1(n)2$. Since $\psi_1(n)$ ends with $R_2R_1$, then $\psi_1(n)2$ ends with $R_2R_12$. We can computationally verify that $R_2R_12$ generates $R_2R_1\psi_1(0)$, which implies the result by using Remark~\ref{Appending to front of lex-least words}.
\end{proof}

\begin{lemma}\label{adding k+1 in psi1}
For $0\leq k\leq n$, $k+1$ is irreducible at the end of ${\psi_1(nR_k[:-1])(k+1)}$.
\end{lemma}

\begin{proof}
Let $0\le m\le k$, we will show that $w_m:=\psi_1(nR_k[:-1])m$ has a square suffix. If $k=0$, then $m=0$ and $\psi_1(nR_k[:-1])=\psi_1(n)$. Then $\psi_1(nR_k[:-1])m=\psi_1(n)0$ which has a square suffix by Lemma~\ref{psi_1(n) generates psi_1(n0)}.

Now note that for $k\ge 1$, the last letter of $R_k[:-1]$ is $0$ and $\psi_1(0)=202101$, so $w_m$ has a square suffix for $m<2$.

Now assume $2\le m\le k$. By definition $\psi_1(n)=(n+1)P_0(n+1)$, and by Remark~\ref{structure of P0}, $P_0(n+1)$ has suffix $P_0(m)$. Hence, if $m=k$, $w_m$ has suffix $P_0(m)\psi_1(R_{m-1}[:-1])m$.

On the other hand, if $m<k$ it is easy to see that $nR_k[:-1]$ has suffix $mR_m[:-1]$. By definition $\psi_1(m)=(m+1)P_0(m+1)$, and since $m\ge 2$, Remark~\ref{structure of P0} implies that $P_0(m+1)$ has suffix $P_0(m)$. So $w_m$ has suffix $P_0(m)\psi_1(R_m[:-1])m$ for any $2\le m\leq k$. Finally note that
\begin{align*}
    P_0(m)\psi_1(R_m[:-1])m&=P_0(m)\psi_1(R_{m-1}[:-1])\psi_1(m-1)\psi_1(R_{m-1}[:-1])m\\
    &=P_0(m)\psi_1(R_{m-1}[:-1])mP_0(m)\psi_1(R_{m-1}[:-1])m,
\end{align*}
which is a square.
\end{proof}

\begin{lemma}\label{psi_1(n) generates psi_1(nR_k)}
For $n\geq1$ and $0\leq k\leq n$, $\psi_1(n)$ generates $\psi_1(nR_k)$.
\end{lemma}
\begin{proof}
Since $k\le n$ we know that $nR_k$ is square-free and grounded, so Proposition~\ref{psi_1 square-free over grounded} implies that $\psi_1(nR_k)$ is square-free. It is now sufficient to show that $\psi_1(R_k)$ is irreducible in $\psi_1(nR_k)$.

We prove this inductively over $n$. The base case $n=1$ implies that $k=0$ or $k=1$. We can computationally verify that $\psi_1(1)$ generates $\psi_1(1R_0)=\psi_1(10)$ and that $\psi_1(1)$ generates $\psi_1(1R_1)=\psi_1(101)$.

Fix $n>1$ and suppose the result holds for all $1\leq n_0<n$. That is,
\begin{equation}
    \psi_1(R_k)\textrm{ is irreducible in }\psi_1(n_0R_k)\textrm{ for all }0\leq k\leq n_0\textrm{ and }1\leq n_0<n.\tag{i}\label{psi_1(n) generates psi_1(nR_k) induction-1}
\end{equation}
We will show that the result holds for $n_0=n$. That is, $\psi_1(R_k)$ is irreducible in $\psi_1(nR_k)$ for all $0\leq k\leq n$.

We can prove this intermediate step by a second induction, now over $k$. The base case is $k=0$, which holds by Lemma~\ref{psi_1(n) generates psi_1(n0)} since $R_0=0$. Now fix $k>0$ and suppose
\begin{equation}
    \psi_1(R_{k_0})\textrm{ is irreducible in }\psi_1(nR_{k_0})\textrm{ for all }0\leq k_0<k\tag{ii}\label{psi_1(n) generates psi_1(nR_k) induction-2}
\end{equation}
We will show that the result holds for $k_0=k$. That is, $\psi_1(R_k)$ is irreducible in $\psi_1(nR_k)$.

We have that $\psi_1(nR_k)=\psi_1(nR_{k-1}R_{k-2}\cdots R_2R_1R_0k)$. By the second inductive hypothesis~(\ref{psi_1(n) generates psi_1(nR_k) induction-2}), $\psi_1(R_{k-1})$ is irreducible in $\psi_1(nR_{k-1})$. The last letter of $R_{k-1}$ is $k-1$, so the first inductive hypothesis~(\ref{psi_1(n) generates psi_1(nR_k) induction-1}) says that $\psi_1(R_{k-2})$ is irreducible in $\psi_1((k-1)R_{k-2})$ and so $\psi_1(R_{k-1}R_{k-2})$ is irreducible in $\psi_1(nR_{k-1}R_{k-2})$. Repeating this argument shows that $\psi_1(R_{k-1}R_{k-2}\cdots R_2R_1R_0)=\psi_1(R_k[:-1])$ is irreducible in $\psi_1(nR_k[:-1])$. Lemma~\ref{adding k+1 in psi1} implies that $k+1$ is irreducible in $\psi_1(nR_k[:-1])(k+1)$. And $\psi_1(k)=(k+1)P_0(k+1)$ is a prefix of $L(k+1)$ by the definition of $P_0$, so $k+1$ generates $\psi_1(k)$, meaning that $\psi_1(k)$ is irreducible in $\psi_1(nR_k[:-1])\psi_1(k)=\psi_1(nR_k)$, which proves the result.
\end{proof}

In particular, this lemma implies that for all $n\geq1$, $\psi_1(n)$ generates $\psi_1(nR_n)$.

\begin{theorem}\label{P1 theorem}
    For all $n\geq3$, let $P_1(n)=\psi_{1}(P_0(n-1))$. Then for $n\geq3$, $L(n)$ has prefix $n \, P_0(n) \, P_1(n)$.
\end{theorem}

\begin{proof}
Note that $nP_0(n)P_1(n) = \psi_{1}((n-1)P_0(n-1))$. Then, since ${(n-1)P_0(n-1)}$ is square-free and grounded, Proposition~\ref{psi_1 square-free over grounded} implies that $nP_0(n)P_1(n)$ is square-free. It remains to show that $P_1(n)=\psi_1(P_0(n-1))=\psi_1(R_{n-1}R_{n-1}[:-2])$ is irreducible in $nP_0(n)P_1(n)$. Indeed, note that \[nP_0(n)P_1(n)=\psi_1((n-1)P_0(n-1))=\psi_1((n-1)R_{n-1}R_{n-1}[:-2]).\]
Lemma~\ref{psi_1(n) generates psi_1(nR_k)} implies that $\psi_1(n-1)$ generates $\psi_1((n-1)R_{n-1})$, which has suffix $\psi_1(n-1)$. Hence, by applying Lemma~\ref{psi_1(n) generates psi_1(nR_k)} a second time we obtain in particular that $\psi_1(R_{n-1}[:-2])$ is irreducible in $\psi_1((n-1)R_{n-1}[:-2])$, and so it is also irreducible in 
\[nP_0(n)P_1(n)=\psi_1((n-1)R_{n-1}R_{n-1}[:-2]).\]
 Therefore, the whole factor $P_1(n)$ is an irreducible suffix of $nP_0(n)P_1(n)$.
\end{proof}

\begin{remark}\label{structure of P1}
Using Remark~\ref{structure of P0}, for all $n\ge 4$ we have
\begin{align*}
    P_1(n) &= \psi_{1}(P_0(n-1))\\
          &=\psi_{1}(R_{n-1}[:-1](n-1)R_{n-1}[:-2])\\
        &= \psi_{1}(R_{n-1}[:-1]) \, \psi_1(n-1) \, \psi_1(P_0(n-2))\\
        &= \psi_{1}(R_{n-1}[:-1]) \, n \, P_0(n) \,\,\,\,\,\, P_1(n-1)\\
        &=\psi_{1}(R_{n-1}[:-1]) \, n \, R_n P_0(n-1)\, P_1(n-1),
\end{align*}
so $P_1(n)$ has $nP_0(n-1) \, P_1(n-1)$ as suffix.

Repeated application of this argument implies that for all $2\leq k<n$, the word $P_1(n)$ has the suffix $(k+1)P_0(k)P_1(k)$.
\end{remark}

\subsection{The morphism $\psi_2$}

For all $n\geq3$ after the prefix given by the previous result, $L(n)$ continues with another sequence, $P_2(n)$ defined as follows.

\begin{definition*}
Define the morphism $\psi_2 \colon \N^\ast\to\N^\ast$ by
\begin{align*}
	\psi_2(0) = {}
	& 2021020102101201020120210120102013010201030102012021012010201 \\
	& 2021013010201030102012021012010201202301020103010201202101201 \\
	& 0201203010201030102030103020102030102010301020301030201202101 \\
	& 2010201202101202,\\
	\psi_2(n) = {}
	& (n+2)P_0(n+2)P_1(n+2),\text{ if }n>0
\end{align*}
and let
\[P_2(n)=\psi_2(P_0(n-2))\]
which can be referenced in Section~\ref{glossary}.
\end{definition*}
In particular, we note that for $n\geq1$, $\psi_2(n)=\psi_1((n+1)P_0(n+1))=\psi_1^2(n)$, but $\psi_2(0)\neq\psi_1^2(0)$. Similarly to $\psi_1$, we would like to show that $\psi_2$ is square-free over grounded words because that will imply that $nT(n)=\psi_2((n-2)P_0(n-2))$ is square-free. The next four lemmas prove some properties of $\psi_2$ that are used to prove this condition.

\begin{lemma}\label{psi_2(n0) square-free}
Let $n\geq1$. Then $\psi_2(n0)$ is square-free.
\end{lemma}
\begin{proof}
We can verify computationally that $\psi_2(10)$ is square-free. So let $n\geq2$ and suppose there is a square $yy$ in $\psi_2(n0)$. Since $\psi_2(n)=\psi_1((n+1)P_0(n+1))$, $\psi_2(n)$ is square-free by Proposition~\ref{psi_1 square-free over grounded}. This implies that $yy$ overlaps both chunks. We have that $\psi_2(n)$ ends with
\[P_1(n+2)=\psi_1(P_0(n+1))=\psi_1(R_{n+2}[:-2])=\psi_1(R_{n+1}R_n\cdots R_3R_2R_1).\]
Since $n\geq2$, this ends with $\psi_1(3R_2R_1)$. We can computationally verify that $\psi_1(3R_2R_1)\psi_2(0)$ is square-free which means that any square in $\psi_2(n0)$ must contain all of $\psi_1(3R_2R_1)][2$ at the chunk boundary. Since $\psi_1(3)$ contains 4's, the square must contain 4's.

Let $k$ be the largest letter in the square $yy$. We know that $4\leq k\leq n+2$ since $\max(\psi_2(n0))=n+2$. Since $\max(\psi_2(0))=3$, all occurrences of $k$ in $\psi_2(n0)$ are in the $n$-chunk. From above, $\psi_2(n)$ ends with $\psi_1(R_m[:-2])$ for $m\leq n+2$. We consider two cases:

\textbf{Case 1: }$k<n+2$. Since $k+1\leq n+2$, $\psi_2(n)$ ends with
\begin{align*}
    \psi_1(R_{k+1}[:-2])&=\psi_1(R_k[:-1])&&\psi_1(k)&&\psi_1(R_k[:-2])\\
    &=\psi_1(R_k[:-1])&&\underline{(k+1)}R_{k+1}[:-1]\underline{(k+1)}R_{k+1}[:-2]&&\psi_1(R_k[:-2])
\end{align*}
which contains the last two occurrences of $k+1$ in $\psi_2(n)$. Since $k$ is the largest letter in the square, $yy$ must be contained in the suffix of $\psi_2(n0)$ after the last occurrence of $k+1$ which is
\begin{align*}
&R_{k+1}[:-2]&&\psi_1(R_k[:-2])&&&&\psi_2(0)\\
={}&R_kR_k[:-2]&&\psi_1(R_{k-1}R_{k-1}[:-2])&&&&\psi_2(0)\\
={}&R_k[:-1]\ k\ R_k[:-2]&&\psi_1(R_{k-1}[:-1])\psi_1(k-1)&&\psi_1(R_{k-1}[:-2])&&\psi_2(0)\\
={}&R_k[:-1]\ \underline{k}\ R_k[:-2]&&\psi_1(R_{k-1}[:-1])\ \underline{k}\ R_k[:-1]\ \underline{k}\ R_k[:-2]&&\psi_1(R_{k-1}[:-2])&&\psi_2(0)
\end{align*}
These are the only three occurrences of $k$ that can occur in $yy$. The square must contain an even number of $k$ occurrences, so since $yy$ overlaps part of $\psi_2(0)$, it cannot include the first occurrence of $k$. Therefore, the square contains only the last two occurrences of $k$. Since the last letter of $R_m[:-1]$ is 0 for all $m$, the second last occurrence of $k$ is preceded by $\psi_1(0)$ which ends in 1, and the last occurrence of $k$ is preceded by a 0. This means that $k$ must be the first letter of $y$ in the square. It can be seen from the above equation that this square would not reach $\psi_2(0)$ which is a contradiction.

\textbf{Case 2: }$k=n+2$. We have that
\begin{align*}
    \psi_2(n)&=\psi_1(n+1)\cdot \psi_1(P_0(n+1))\\
    &=\psi_1(n+1)\cdot \psi_1(R_{n+1}R_{n+1}[:-2])\\
    &=\psi_1(n+1)\cdot\psi_1(R_{n+1}[:-1])\cdot\psi_1(n+1)\cdot\psi_1(R_{n+1}[:-2])\\
    &=(n+2)P_0(n+2)\cdot\psi_1(R_{n+1}[:-1])\cdot(n+2)P_0(n+2)\cdot \psi_1(R_{n+1}[:-2]).
\end{align*}
And so $\psi_2(n)$ is the ``almost square"
\begin{align}\label{psi_2(n) decomposition}
\psi_2(n)={}&\underline{(n+2)}\ \underline{R_{n+2}}\ R_{n+2}[:-2]\psi_1(R_{n+1}[:-1])\notag\\
{}&\underline{(n+2)}\ \underline{R_{n+2}}\  R_{n+2}[:-2] \psi_1(R_{n+1}[:-2]).
\end{align}

and these are all four of the occurrences of $n+2$ in $\psi_2(n0)$. The square must contain an even number of occurrences of $n+2$.

The square cannot contain all four occurrences of $n+2$ since then $\psi_2(0)$ would need to begin with $\psi_1(R_{n+1}[-2])=\psi_1(0)$ which it does not.

So the square only contains the last two occurrences of $n+2$. The second last occurrence is preceded by $\psi_1(0)$ which ends in a 1, and the last occurrence is preceded by a 0. This means that $n+2$ must be the first letter of $y$ in the square. It can be seen from the equation for $\psi_2(n)$ that this square would not reach $\psi_2(0)$ which is a contradiction.
\end{proof}

\begin{lemma}\label{psi_2(0n0) square-free}
Let $n\geq1$. Then $\psi_2(0n0)$ is square-free.
\end{lemma}
\begin{proof}
We can verify computationally that $\psi_2(010)$ is square-free. So let $n\geq2$ and suppose there is a square $yy$ in $\psi_2(0n0)$. We have
\[\psi_2(0n0)=[\psi_2(0)][(n+2)P_0(n+2)P_1(n+2)][\psi_2(0)]\]
Since $\psi_2(n)$ is square-free, $yy$ lies over at least one of the chunk boundaries.

The largest letter of $\psi_2(0)$ is 3, so the largest letter of $\psi_2(0n0)$ is $n+2$ which occurs exactly four times:
\[\psi_2(n)=\underline{(n+2)}\ \underline{R_{n+2}}\ R_{n+2}[:-2]\psi_1(R_{n+1}[:-1])\underline{(n+2)}\ \underline{R_{n+2}}\ R_{n+2}[:-2]\psi_1(R_{n+1}[:-2])\]
The second and fourth occurrences of $n+2$ are the last letter of an $R_{n+2}$, so they are preceded by a 0. The third occurrence is after $\psi_1(R_{n+1}[:-1])$ which ends with $\psi_1(0)=202101$.

The first chunk boundary in $\psi_2(0n0)$ has the letters $2(n+2)$. The second, third, and fourth occurrences of $n+2$ all immediately follow a 0 or a 1, so this is the only occurrence of $2(n+2)$ in $\psi_2(0n0)$. This means that it cannot be contained in either half of the square. Since $\max(\psi_2(0))=3$ and $n+2\geq4$, $2(n+2)$ cannot be the middle of the square either. So the square is contained in $\psi_2(n0)$ which contradicts Lemma~\ref{psi_2(n0) square-free}.
\end{proof}

\begin{lemma}\label{psi_2(n) and psi_2(k) not suffix}
For $n>k>0$, neither $\psi_2(n)$ nor $\psi_2(k)$ is a suffix of the other.
\end{lemma}
\begin{proof}
Since $n>k$, $|\psi_2(n)|>|\psi_2(k)|$, so $\psi_2(n)$ cannot be a suffix of $\psi_2(k)$.

From the definition of $\psi_2$, $\psi_2(k)={(k+2)P_0(k+2)P_1(k+2)}$ and $\psi_2(n)$ ends with $P_1(n+2)$. By Remark~\ref{structure of P1}, $P_1(n+2)$ has suffix $(k+3)P_0(k+2)P_1(k+2)$. Therefore, $\psi_2(k)$ cannot be a suffix of $\psi_2(n)$.
\end{proof}

\begin{lemma}\label{psi_2(0) is a 0-chunk}
Let $w$ be a grounded square-free word. Then in $\psi_2(w)$, every occurrence of $\psi_2(0)[:6]$ is a prefix of a $0$-chunk, and every occurrence of $\psi_2(0)[6:]$ is a suffix of a $0$-chunk.
\end{lemma}
\begin{proof}
Let $p:=\psi_2(0)[:6]=202102$ and $s:=\psi_2(0)[-9:]=202101202$. Since $s$ is a suffix of $\psi_2(0)[6:]$, proving the result for $p$ and $s$ is sufficient.

For $n\geq1$, $\psi_2(n)$ ends with 1 and $\psi_2(0)$ begins with 2 and $12$ does not occur in $p=202102$, so $p$ cannot lie over a $\psi_2(n0)$ chunk boundary. The first letter of $\psi_2(n)$ is $n+2\geq3$ which is larger than any letter in $p$, so $p$ cannot lie over a $\psi_2(0n)$ chunk boundary. Therefore, any occurrence of $p$ in $\psi_2(w)$ must be contained in a single chunk.

Since $\psi_1(n)=(n+1)R_{n+2}[:-2]$ does not contain $202$, $p$ does not occur in any $\psi_1(n)$. Also, $\psi_1(n)$ ends with $201$, $\psi_2(n)=\psi_1((n+1)P_0(n+1))$ and $(n+1)P_0(n+1)$ is grounded, $202$ only occurs in $\psi_2(n)$ at the beginning of instances of $\psi_1(0)=202101$. This means that $p=202102$ can never occur in $\psi_2(n)$. We can verify computationally that $202102$ only occurs in $\psi_2(0)$ as a prefix. Therefore, $p$ only occurs in $\psi_2(w)$ as a prefix of 0-chunks.

If $s=202101202$ lies over a $\psi_2(n0)$ chunk boundary, the then the $12$ would need to be at the boundary. But this cannot happen since no $\psi_2(n)$ ends with $202101$. It also cannot lie over a $\psi_2(0n)$ boundary since the first letter of $\psi_2(n)$ is $n+2\geq3$ which does not occur in $s$.

By Lemma~\ref{psi_1(l) is a l-chunk}, $\psi_1(0)=202101$ only occurs in $\psi_2(n)=\psi_1((n+1)P_0(n+1))$ as an 0-chunk. So since no $\psi_1(n)$ begins with $202$, $s=202101202$ does not occur in $\psi_2(n)$. We can verify computationally that $s$ only occurs in $\psi_2(0)$ as a suffix. Therefore, $s$ only occurs in $\psi_2(w)$ as a suffix of 0-chunks.
\end{proof}

We can easily see that this lemma implies that any occurrence of $\psi_2(0)$ in $\psi_2(w)$ is a 0-chunk when $w$ is square-free and grounded.

\begin{proposition}\label{psi_2 square-free over grounded}
$\psi_2$ is square-free over grounded words.
\end{proposition}
\begin{proof}
Suppose $w$ is a grounded square-free word and that $\psi_2(w)$ contains a square $yy$. We will first show that there is a whole 0-chunk in both halves of the square, and then use that to show that both halves have the same chunk decomposition. We then show that if the square contains any partial chunks, then the final partial chunks of the two halves of the square come from the same letter in $w$. We know that $\psi_2$ is letter injective and $\psi_2(n)$ is square-free for all $n$. Hence, Theorem~\ref{chunk-identifiable condition for square-freeness} will imply that $w$ must contain a square, which is a contradiction.

Suppose that there are no whole 0-chunks in either half. The square must contain a whole 0-chunk, or else it would be a factor of $\psi_2(0n0)$, contradicting Lemma~\ref{psi_2(0n0) square-free}. Then the whole 0-chunk in the square must be split between the two halves and the square $yy$ is a proper factor of $\psi_2(0n0k0)$ for some $n,k\geq1$ with the center of the square lying in the middle 0-chunk. Consider the prefix $p:=\psi_2(0)[:6]$ and suffix $s:=\psi_2(0)[6:]$ of $\psi_2(0)=ps$. Then either $p$ is totally contained in the first half of the square, or $s$ is totally contained in the second half.

If $s$ is contained in the second half, then it must also occur in the first half and by Lemma~\ref{psi_2(0) is a 0-chunk}, the only place for this is as a suffix of the first 0-chunk in $\psi_2(0n0k0)$. The two occurrences of $s$ in the square are followed by the first letters of $\psi_2(n)$ and $\psi_2(k)$ in each half respectively, so the first letter of $\psi_2(n)$, $n+2$ and of $\psi_2(k)$, $k+2$ must be equal. This means that $n=k$ which is a contradiction since $w$ is square-free and cannot contain $0n0n0$.

If $p$ is contained in the first half, then it must also occur in the second half and by Lemma~\ref{psi_2(0) is a 0-chunk}, the only place for this is as a prefix of the third 0-chunk in $\psi_2(0n0k0)=[ps][\psi_2(n)][ps][\psi_2(k)][ps]$. The center of the square lies in the suffix $s$ of the middle 0-chunk, so the boundary is formed by words $s_1,s_2$ such that $s=s_1s_2$. We then get from the first half of the square that $y$ is a proper suffix of $ps\psi_2(n)ps_1$ and from the second half that $y$ is a prefix of $s_2\psi_2(k)ps$. Since $p$ occurs exactly once in each half, we can see that the second half is $y=s_2\psi_2(k)ps_1$. Then $s_2\psi_2(k)$ is a suffix of $ps\psi_2(n)$ implying that one of $\psi_2(n)$ or $\psi_2(k)$ must be a suffix of the other, so by Lemma~\ref{psi_2(n) and psi_2(k) not suffix}, $n=k$ which is a contradiction since $w$ is square-free and cannot contain $0n0n0$. Therefore, there must be a whole 0-chunk in one of the halves of the square. Then $\psi_2(0)$ occurs in both halves and by Lemma~\ref{psi_2(0) is a 0-chunk}, both halves contain a whole 0-chunk.

We now show that both halves of the square have the same chunk decomposition. Let $[\psi_2(\ell)],\ 0\leq\ell$ be any whole chunk in either half of the square. If $l=0$, then by Lemma~\ref{psi_2(0) is a 0-chunk}, this is a whole 0-chunk in both halves. If $\ell>0$, then since there is a whole 0-chunk in each half and $w$ is grounded, there must be a whole 0-chunk adjacent to this chunk. Thus, either $[ps]\psi_2(\ell)$ or $\psi_2(\ell)[ps]$ is a factor of $y$. If $[ps]\psi_2(\ell)$ is a factor of $y$, then $\psi_2(\ell)$ must be a whole $\ell$-chunk in both halves since any other chunk would start with a different letter. If $\psi_2(\ell)[ps]$ is a factor, then $\psi_2(\ell)$ must be a whole $l$-chunk in both halves since no other chunk can be a suffix of $\psi_2(\ell)$ or have $\psi_2(\ell)$ as a suffix by Lemma~\ref{psi_2(n) and psi_2(k) not suffix}. Thus, $\psi_2(\ell)$ is a whole chunk in both halves of the square, so both halves have the same chunk decomposition.

Suppose either half contains a partial chunk. Then since the halves have the same chunk decomposition, they must both end with a partial chunk and share the first letter of their final partial chunk, say $\ell$. From the definition of $\psi_2$, this means that both final partial chunks are partial $(\ell-2)$-chunks.

This verifies the conditions of Theorem~\ref{chunk-identifiable condition for square-freeness}  which implies that $w$ contains a square, a contradiction.
\end{proof}

The next three lemmas are used to prove the irreducibility condition in Theorem~\ref{tau prefix theorem}. They are analogous to Lemmas~\ref{psi_1(n) generates psi_1(n0)},~\ref{adding k+1 in psi1}, and~\ref{psi_1(n) generates psi_1(nR_k)} about $\psi_1$.

\begin{lemma}\label{psi_2(n) generates psi_2(n0)}
Let $n\geq1$. Then $\psi_2(n)$ generates $\psi_2(n0)$.
\end{lemma}
\begin{proof}
Since $n0$ is square-free and grounded, $\psi_2(n0)$ is square-free by Proposition~\ref{psi_2 square-free over grounded}, so we only need to show that $\psi_2(0)$ is irreducible in $\psi_2(n0)$. We can computationally verify that $\psi_2(1)$ generates $\psi_2(10)$ and assume that $n\geq2$.

Since $\psi_2(n)$ has suffix $P_1(n+2)$ and $n+2\geq4$, Remark~\ref{structure of P1} implies that $\psi_2(n)$ has suffix $P_0(3)P_1(3)$. We can computationally verify that $P_0(3)P_1(3)$ generates $P_0(3)P_1(3)20210$. The next letter of $\psi_2(0)$ is a 2. Hence, we need to show that $2$ is irreducible at the end of $\psi_2(n)202102$. Clearly, it cannot be a 0, so we will show that $\psi_2(n)202101$ contains a square. Using Equation~\eqref{psi_2(n) decomposition} we have
\begin{align*}
    &\psi_2(n)\ 202101\\
    ={}&\psi_2(n)\ \psi_1(0)\\
    ={}&(n+2)R_{n+2}R_{n+2}[:-2]\psi_1(R_{n+1}[:-1])(n+2)R_{n+2}R_{n+2}[:-2]\psi_1(R_{n+1}[:-2])\ \psi_1(0)\\
    ={}&(n+2)R_{n+2}R_{n+2}[:-2]\psi_1(R_{n+1}[:-1])(n+2)R_{n+2}R_{n+2}[:-2]\psi_1(R_{n+1}[:-1]),
\end{align*}
which is a square, so $202102$ is irreducible in $\psi_2(n)202102$. We can then computationally verify that $P_0(3)P_1(3)202102$ generates $P_0(3)P_1(3)\psi_2(0)$, which implies the desired result.
\end{proof}

\begin{lemma}\label{adding k+2 in psi2}
For $0\leq k\leq n$, $k+2$ is irreducible at the end of ${\psi_2(nR_k[:-1])(k+2)}$.
\end{lemma}

\begin{proof}
Let $0\le m\le k+1$, we will show that $w_m=\psi_2(nR_k[:-1])m$ has a square suffix. If $k=0$, then $m\leq1$. For $m=0$ we can use Remark~\ref{structure of P0} as in Lemma~\ref{adding k+1 in psi1} to show that $\psi_2(nR_k[:-1])0=\psi_2(n)0$ is a square, and for $m=1$,  $\psi_2(n)1$ has the square suffix $11$. Hence, $k+2=2$ is irreducible.

Note that for $k\ge 1$ the last letter of $R_k[:-1]$ is 0 and $\psi_2(0)$ ends with $01202101202$, so $w_m$ has a square suffix for $m<3$. Assume $m\geq3$, by definition $\psi_2(n)=(n+2)P_0(n+2)P_1(n+2)$, and according to Remark~\ref{structure of P1} $P_1(n+2)$ has suffix $P_0(m)P_1(m)$. 

If $m-1=k$, then $w_m$ has suffix $P_0(m)P_1(m)\psi_2(R_{m-1}[:-1])m$. Also, if $m-1<k$ it is easy to see that $nR_k[:-1]$ has suffix $(m-1)R_{m-1}[:-1]$. By definition $\psi_2(m-1)=(m+1)P_0(m+1)P_1(m+1)$, and again Remark~\ref{structure of P1} implies that $P_1(m+1)$ has suffix $P_0(m)P_1(m)$. Hence, $w_m$ has suffix $P_0(m)P_1(m)\psi_2(R_{m-1}[:-1])m$ for all $3\le m\le k+1$. Finally, we have that 
\begin{align*}
    &P_0(m)P_1(m)\psi_2(R_{m-1}[:-1])m\\
    ={}&P_0(m)P_1(m)\psi_2(R_{m-2}[:-1])\psi_2(m-2)\psi_2(R_{m-2}[:-1])m\\
    ={}&P_0(m)P_1(m)\psi_2(R_{m-2}[:-1])m\ P_0(m)P_1(m)\psi_2(R_{m-2}[:-1])m,
\end{align*}
which is a square.
\end{proof}

\begin{lemma}\label{psi_2(n) generates psi_2(nR_k)}
If $n\geq1$ and $0\leq k\leq n$, then $\psi_2(n)$ generates $\psi_2(nR_k)$.
\end{lemma}
\begin{proof}
Since $nR_k$ is square-free and grounded, Proposition~\ref{psi_2 square-free over grounded} implies that $\psi_2(nR_k)$ is square-free. It is now sufficient to show that $\psi_2(R_k)$ is irreducible in $\psi_2(nR_k)$.

We prove this inductively over $n$. The base case $n=1$ implies that $k=0$ or $k=1$. We can computationally verify that $\psi_2(1)$ generates $\psi_2(1R_0)=\psi_2(10)$ and that $\psi_2(1)$ generates $\psi_2(1R_1)=\psi_2(101)$.

Fix $n>1$ and suppose the result holds for all $1\leq n_0<n$. That is,
\begin{equation}
    \psi_2(R_k)\textrm{ is irreducible in }\psi_2(n_0R_k)\textrm{ for all }0\leq k\leq n_0\textrm{ and }1\leq n_0<n.\tag{i}\label{psi_2(n) generates psi_2(nR_k) induction-1}
\end{equation}
We will show that the result holds for $n_0=n$. That is, $\psi_2(R_k)$ is irreducible in $\psi_2(nR_k)$ for all $0\leq k\leq n$.

We can prove this intermediate step by a second induction, now over $k$. The base case is $k=0$, which holds by Lemma~\ref{psi_2(n) generates psi_2(n0)} since $R_0=0$. Now fix $k$ and suppose
\begin{equation}
    \psi_2(R_{k_0})\textrm{ is irreducible in }\psi_2(nR_{k_0})\textrm{ for all }0\leq k_0<k.\tag{ii}\label{psi_2(n) generates psi_2(nR_k) induction-2}
\end{equation}
We will show that the result holds for $k_0=k$. That is, $\psi_2(R_k)$ is irreducible in $\psi_2(nR_k)$.

We have that $\psi_2(nR_k)=\psi_2(nR_{k-1}R_{k-2}\cdots R_2R_1R_0k)$. By the second inductive hypothesis~(\ref{psi_2(n) generates psi_2(nR_k) induction-2}), $\psi_2(R_{k-1})$ is irreducible in $\psi_2(nR_{k-1})$. The last letter of $R_{k-1}$ is $k-1$, so the first inductive hypothesis~(\ref{psi_2(n) generates psi_2(nR_k) induction-1}) says that $\psi_2(R_{k-2})$ is irreducible in $\psi_2((k-1)R_{k-2})$ and so $\psi_2(R_{k-1}R_{k-2})$ is irreducible in $\psi_2(nR_{k-1}R_{k-2})$. Repeating this argument shows that $\psi_2(R_{k-1}R_{k-2}\cdots R_2R_1R_0)=\psi_2(R_k[:-1])$ is irreducible in $\psi_2(nR_k[:-1])$. Lemma~\ref{adding k+2 in psi2} implies that $k+2$ is irreducible in $\psi_2(nR_k[:-1])(k+2)$. And $\psi_2(k)=(k+2)P_0(k+2)P_1(k+2)$ is a prefix of $L(k+2)$ by Theorem~\ref{P1 theorem}, so $k+2$ generates $\psi_2(k)$, meaning that $\psi_2(k)$ is irreducible in $\psi_2(nR_k[:-1])\psi_2(k)=\psi_2(nR_k)$, which proves the result.
\end{proof}

We now prove the main theorem of this section. For $n\geq3$, define
\[T(n)=P_0(n)P_1(n)P_2(n).\]

\begin{theorem}\label{tau prefix theorem}
	For $n\geq3$, $L(n)$ has prefix $nT(n)$.
\end{theorem}

\begin{proof}
Note that $nT(n)=nP_0(n)P_1(n)P_2(n) = \psi_{2}((n-2)P_0(n-2))$. Hence, since $(n-2)P_0(n-2)$ is square-free and grounded, Proposition~\ref{psi_2 square-free over grounded} implies that $nT(n)$ is square-free. It remains to show that $P_2(n)=\psi_2(P_0(n-2))=\psi_2(R_{n-2}R_{n-2}[:-2])$ is irreducible in $nP_0(n)P_1(n)P_2(n)$. 

We know from Theorem~\ref{P1 theorem} that $n$ generates $nP_0(n)P_1(n)$. The fact that $P_2(n)=\psi_{2}(R_{n-2}R_{n-2}[:-2])$ is irreducible in $nP_0(n)P_1(n)P_2(n)$ follows from Lemma~\ref{psi_2(n) generates psi_2(nR_k)} by the same argument used in the proof of Theorem~\ref{P1 theorem}.
\end{proof}

\begin{remark}\label{structure of P2 and tau(n) has suffix tau(3)}
Again, using Remark~\ref{structure of P0} we have that for $n\ge4$
\begin{align*}
    P_2(n) &= \psi_2(P_0(n-2))\\
        &=\psi_2(R_{n-2}[:-1](n-1)R_{n-2}[:-2])\\
        &= \psi_2(R_{n-2}[:-1]) \ \psi_2(n-2) \ \psi_2(P_0(n-3))\\
        &= \psi_2(R_{n-2}[:-1]) \ n \, P_0(n) \, P_1(n) \ P_2(n-1),
\end{align*}
Hence, by Remark~\ref{structure of P1} we see that $P_2(n)$ has $T(n-1)=P_0(n-1) \,P_1(n-1) \, P_2(n-1)$ as a suffix. This means that $T(n)$ has suffix $T(n-1)$, and applying the same argument repeatedly we have that $T(n)$ has $T(3)$ as a suffix.
\end{remark}
%%%%%%%%%%%%%%%%%%%%%%%%%%%%%%%%
\section{The structure of $L(n)$}\label{section: one-letter prefixes}

In this section we prove that the word $L(n)$ reflects the structure of the ruler sequence for $n = 1$ and $n \geq 3$.
Namely, $L(n) = Y_n \, \phi(L(\varepsilon))$ for a finite prefix $Y_n$ and a morphism $\phi$.

\begin{definition*}
We say a morphism $\phi$ is \emph{$L$-commuting} over a set of words $\Delta \subset \N^*\cup \N^{\omega}$ if $L(\phi(w)) = \phi(L(w))$ for all $w \in \Delta$.
\end{definition*}

Let $\Sigma$ be the set of all nonempty even-grounded square-free words. If $\phi$ is $L$-commuting over $\Sigma$, we get in particular $$L(\phi(0)) =\phi(L(0))= \phi(L(\varepsilon)),$$ 
which lets us determine the lexicographically least square-free word with prefix $\phi(0)$, as the result of applying the morphism $\phi$ to the ruler sequence.

In Section~\ref{morphism alpha} we introduce a morphism $\alpha$ and prove that it is $L$-commuting over the set $\Sigma$ of even-grounded square-free words. In Section~\ref{L(1) section} we use Remark~\ref{Appending to front of lex-least words} and the $L$-commuting property of $\alpha$ to find the general structure of $L(1)$ and $L(n)$ for $n\ge 3$. In Section~\ref{L(2) section} we state a conjecture about the structure of $L(2)$ being given by a morphism $\gamma$.

We start by showing that the ruler morphism $\rho$ is $L$-commuting over the set of square-free words, and then use this fact to prove a result that establishes properties that are sufficient for a morphism to satisfy the $L$-commuting property over the set $\Sigma$.

\begin{theorem}\label{rho is L-commuting}
The ruler morphism $\rho$ is $L$-commuting over the set of all nonempty square-free words.
\end{theorem}

\begin{proof}
Let $x$ be any nonempty square-free word. If $|x|=\infty$, then $\rho(L(x))= \rho(x)=L(\rho(x))$.

On the other hand, if $|x|=n<\infty$, let $w=\rho(L(x))$ and $v=L(\rho(x))$. We proceed to prove that $w=v$ by induction, proving that if $w$ and $v$ agree on the first $2k$ letters, then they will also agree on the next two letters. For the Base case, it is easy to see that $w[:2n]=v[:2n]=\rho(x)$.

For the inductive step, assume that $w$ and $v$ agree on the first $2k$ letters, $k\ge n$. It is clear from the definition of these words that $w$ and $v$ are both square-free, and if $v\neq w$, then $v\prec w$.

This implies that $v[2k]\le w[2k]=\rho(L(x))[2k]=0$, so $w$ and $v$ agree at position $2k$. Now suppose toward a contradiction that $v[2k+1]<w[2k+1]$ and let $l=v[2k+1]$ and $y= L(x)[:k](l-1)$. We have
\[\rho(y)=\rho(L(x)[:k](l-1))=\rho(L(x)[:k])\,0\,l=L(\rho(x))[:2k+2]=v[:2k+2],\]
so $y$ is square-free. However \[y[k]=l-1=v[2k+1]-1<w[2k+1]-1=\rho(L(x))[2k+1]-1=L(x)[k].\]
Since $y[:k]=L(x)[:k]$, this implies that $y$ is a square-free word beginning with $x$ that is smaller than $L(x)$. This is a contradiction and so $w[2k+1] = v[2k+1]$. Therefore $w$ and $v$ agree at position $2k+1$, which proves the inductive step. 
\end{proof}

For example, Theorem~\ref{rho is L-commuting} implies that for $n\ge 1$, $\rho(L(n-1))=L(\rho(n-1))=L(0n)$. So if we determine the structure of the word $L(n-1)$, this Theorem gives us the structure of the word $L(0n)$ as the ruler morphism applied to $L(n-1)$. In particular for $n=1$ we have 
\[\rho(L(0))=\rho(\rho^{\infty}(0))=\rho^{\infty}(0)=L(01)=L(\rho(0)).\]
The ruler morphism is not $L$-commuting over the set of all words on $\N$. For example, \[L(\rho(00)) = L(0101) = 01012010\cdots \neq 01010201\cdots = \rho(0010\cdots) = \rho(L(00)).\]

\begin{theorem}\label{L-commuting over grounded words template}
Let $\Sigma$ be the set of all nonempty even-grounded square-free words.
Let $\phi$ be a non-erasing morphism satisfying the following conditions.
\begin{enumerate}
    \item For all $w\in \Sigma$, $\phi(w)$ is square-free.\label{L-commuting over grounded condition square-free}
    \item $\phi(0)$ generates $\phi(01)$.\label{L-commuting over grounded condition 0 to 01}
    \item $\phi(0n)$ generates $\phi(0n0)$ for all $n>0$.\label{L-commuting over grounded condition 0n to 0n0}
    \item $\phi(0n)^+$ generates $\phi(0 \, (n+1))$ for all $n>0$.\label{L-commuting over grounded condition 0n-plus to 0n+1}
\end{enumerate}
Then $\phi$ is $L$-commuting over $\Sigma$.
\end{theorem}

\begin{proof}
Let $w\in \Sigma$. We first show that $L(w)\in\Sigma$. If $w=0$, then $L(w)$ is the ruler sequence which is even-grounded and square-free. If $w\neq0$, $|w|\geq2$ and any even-length prefix of $w$ is the image of a nonempty square-free word under $\rho$. So there is a nonempty square-free word $w_0$ such that if $|w|$ is even, $w=\rho(w_0)$, and if $|w|$ is odd, $w[:-1]=\rho(w_0)$. If $|w|$ is odd, its last letter is 0 which is irreducible, so $L(w)=L(w[:-1])$. Thus in either case, $L(w)=L(\rho(w_0))=\rho(L(w_0))$ which is even-grounded and square-free. Here, the second equality is because $\rho$ is $L$-commuting over square-free words by Theorem~\ref{rho is L-commuting}. Since $L(w)\in\Sigma$, it follows from Condition~\ref{L-commuting over grounded condition square-free} that $\phi(L(w))$ is square-free. 

Now we need to show that $\phi(L(w))$ is irreducible. We proceed by induction, assume that for some number $k \geq |w|$, $\phi(L(w)[:k+1])$ is a prefix of $L(\phi(w))$. We break this next bit down into cases, letting $m = L(w)[k]$.

Case 1:
If $m \neq 0$ then $L(w)[k-1] = 0$ and so $L(w)[:k+1]$ ends in $0m$. In this case we have from Condition~\ref{L-commuting over grounded condition 0n to 0n0} that $\phi(L(w)[:k+1])$ generates $\phi(L(w)[:k+1] \, 0)$ by Remark~\ref{Appending to front of lex-least words}. As $L(w)[:k+2] = L(w)[:k+1] \, 0$ this would mean that $\phi(L(w)[:k+2])$ is a prefix of $L(\phi(w))$, demonstrating the inductive step.

Case 2:
In the case that $m = 0$ it must follow that $l := w[k+1] \neq 0$. We have that $\phi(L(w)[:k+1])$ ends with $\phi(0)$ which by Condition~\ref{L-commuting over grounded condition 0 to 01} we get that $\phi(1)$ at the end of $\phi(L(w)[:k+1] \, 1)$ is irreducible. If $l = 1$ then we're done. If not then we enter the following argument.

Let $n$ be such that $\phi(n)$ at the end of $\phi(L(w)[:k+1] \, n)$ is irreducible and $0 < n < l$. $L(w)[:k+1] \, n$ has $w$ as prefix and is lexicographically less than $L(w)$ so $L(w)[:k+1] \, n$ contains a square and therefore $\phi(L(w)[:k+1] \, n)$ contains a square. We have that $\phi(0n)$ at the end of $\phi(L(w)[:k+1]n)$ is irreducible meaning that all words lexicographically less than $\phi(0n)$ would introduce a square. But $\phi(0n)$ also introduces a square. This means that all words less than $\phi(0n)^+$ introduce a square, so $\phi(0n)^+$ is irreducible in $\phi(L(w)[:k+1]n)^+$. Then from Condition~\ref{L-commuting over grounded condition 0n-plus to 0n+1} we get that $\phi(0(n+1))$ at the end of $\phi(L(w)[:k+1] \, (n+1))$ is irreducible. Then by induction on $n$, $\phi(L(w)[:k+1]l)$ is irreducible.\\

This argument allows us to show that $L(\phi(w))$ and $\phi(L(w))$ agree on their $(k+1)$th chunk and so provides the inductive step. The base case is simply when $k = |w|$ which is trivial since $\phi(w)$ is a prefix of both words.
\end{proof}

%%%%%%%%%%%%%%%%%%%%%%%%%%%%%%%%
\subsection{The morphism $\alpha$}\label{morphism alpha}

The morphism $\alpha$ is defined as follows.

\begin{definition*}
For all $n \geq 0$, let
\[\alpha(n) = \begin{cases} 
    EFE & \text{if $n=0$} \\
	B_1 \: R_4 \: C \: B_1 \: R_4 & \text{if $n=1$} \\
	\alpha(n-1)^+ \: R_{n+3} \: C \: \alpha(n-1)^+\: R_{n+3} & \text{if $n \geq 2$},
\end{cases}
\]
where
\begin{align*}
    C &= 0102030102, \\
    B_0 &= 0301 \; \psi_1(1010)[:-3] \; \psi_2(1010)[:-6] \; \psi_2(10)[:-12] \; 301020 ,\\
    B_1 &= \rho(B_0[7:-5]),\\
    E &= 0102B_01B_0[:-9],\\
    F &= B_0[-9:] 3010302 C 0103 C^+ 02,\; \text{and} \\
    G &= 010203012. 
\end{align*}
These definitions can be referenced in Section~\ref{glossary}. The lengths of the auxiliary words are $|C| = 10$, $|B_0| = 798$, $|B_1| = 1572$, $|E| = 1592$, $|F| = 42$ and $|G| = 9$. We can computationally verify that $E$ is the largest word that is both a prefix and a suffix of $\alpha(0)$. It is of interest to note that for $n>0$, $\alpha(n)$ has prefix $B_1$ and that $B_1$ has prefix $F^{++}$, so $F$ cannot be a prefix of any $\alpha(n)$. Also, the word $G$ is useful since it is the shortest word that generates $\alpha(0)$.
\end{definition*}

Over the next two subsections, we prove that $\alpha$ satisfies the conditions of Theorem~\ref{L-commuting over grounded words template}, which will imply that it is $L$-commuting over nonempty even-grounded square-free words.

%%%%%%%%%%%%%%%%%%%%%%%%%%%%%%%%
\subsubsection{Condition~\ref{L-commuting over grounded condition square-free}: $\alpha$ is square-free over grounded words}

Condition~\ref{L-commuting over grounded condition square-free} of Theorem~\ref{L-commuting over grounded words template} says that $\alpha$ is square-free over even-grounded words. In this section, we prove the following stronger property.

\begin{theorem}\label{alpha square-free over grounded words}
$\alpha$ is square-free over grounded words.
\end{theorem}

Theorem~\ref{alpha square-free over grounded words} will be shown to follow from Theorem~\ref{chunk-identifiable condition for square-freeness}. This requires the results about $\alpha$ shown in the following lemmas.

\begin{lemma}\label{alpha shared prefix and last letter}% formerly "flipped" condition 2
	Let $n>k>0$. Then $\alpha(n)$ and $\alpha(k)$ end with different letters and neither is a prefix of the other.
\end{lemma}
\begin{proof}
	By definition, $\alpha(n)$ ends with $R_{n+3}$ which has $n+3$ as its last letter, so $\alpha(n)$ and $\alpha(k)$ end with different letters. Consider that since $n>k$, $|\alpha(n)| > |\alpha(k)|$ so $\alpha(n)$ cannot be a prefix of $\alpha(k)$. Also, $\alpha(n)$ begins with $\alpha(k)^+$ so $\alpha(k)$ cannot be a prefix of $\alpha(n)$.
\end{proof}

The next three results show how $E$ can be used to restrict the placement of chunks throughout a word $\alpha(w)$ where $w$ is square-free and grounded.

\begin{lemma}\label{E is prefix or suffix of alpha(0)}
	If $w$ is a grounded square-free word, then every occurrence of $E$ in $\alpha(w)$ is a prefix or suffix of a $0$-chunk.
\end{lemma}
\begin{proof}
	We can verify computationally that $E$ only occurs in $\alpha(0)$ as a prefix and a suffix. For $n>0$, $\alpha(n)$ is even-grounded, but $E$ is not grounded so it cannot be a factor of $\alpha(n)$. So any other occurrence of $E$ in $\alpha(w)$ must lie over the chunk boundary in $\alpha(0n)$ or $\alpha(n0)$ for some $n>0$.
	
	If $E$ lies over the chunk boundary in $\alpha(0n)$, then there must be a nonempty suffix of $E$ that is also a prefix of $\alpha(n)$. But every prefix of $\alpha(n)$ is even-grounded and $E$ has no nonempty even-grounded suffix since it ends with $1\ 2$.

	If $E$ lies over the chunk boundary in $\alpha(n0)$, then there must be a nonempty prefix of $E$ that is also a suffix of $\alpha(n)$. But $\max(E)=3$ and the last letter of $\alpha(n)$ is $n+3$ which is greater than 3.
\end{proof}

\begin{corollary}\label{alpha(0) is a 0-chunk}
	If $w$ is a grounded square-free word, then every occurrence of $\alpha(0)$ in $\alpha(w)$ is a $0$-chunk.
\end{corollary}
\begin{proof}
	Any occurrence of $\alpha(0)=EFE$ in $\alpha(w)$ begins and ends with $E$. There are only two occurrences of $E$ in $\alpha(0)$ so by Lemma~\ref{E is prefix or suffix of alpha(0)}, one $E$ must be a prefix of a 0-chunk and the other must be a suffix of a 0-chunk. $F$ is shorter than every chunk, so no other chunk can be contained in it. Therefore, this must be a whole 0-chunk.
\end{proof}

\begin{corollary}\label{Ealpha(l) and alpha(l)E indicate l-chunks}
	Let $w$ be a grounded square-free word and $l>0$. If $E\alpha(l)$ or $\alpha(l)E$ is a factor of $\alpha(w)$, then that occurrence of $\alpha(l)$ is an $l$-chunk.
\end{corollary}
\begin{proof}
	If $E\alpha(l)$ occurs in $\alpha(w)$, $E$ is followed by the prefix $F^{++}$ of $\alpha(l)$, so $E$ cannot be followed by $F$ and this cannot be the prefix of a 0-chunk. Thus, $E$ must be a suffix of a 0-chunk by Lemma~\ref{E is prefix or suffix of alpha(0)}. So a nonzero chunk begins at the start of $\alpha(l)$. By Lemma~\ref{alpha shared prefix and last letter}, this must be an $l$-chunk.
		
	An analogous argument uses the fact that $F$ cannot be the suffix of any $\alpha(l)$ to show the result for $\alpha(l)E$.
\end{proof}

The following is the final result that we need for proving Theorem~\ref{alpha square-free over grounded words}, that $\alpha$ is square-free over grounded words.

%%%%%%%%%%%%%%%%%% Third condition
\begin{proposition}\label{alpha(0n0) square-free}
$\alpha(0n0)$ is square-free for all $n>0$.
\end{proposition}

The proof requires Lemmas~\ref{S_1,n suffix} to \ref{alpha(n0) square-free}. Lemmas~\ref{S_1,n suffix} to \ref{S1,n occurrences} show some results about the structure of $\alpha(n)$. Lemmas~\ref{sigma alphan+ square-free} and \ref{alpha(n) square-free} show that $\alpha(n)$ is square-free for all $n$. Finally, Lemmas~\ref{alpha(0n) square-free} and \ref{alpha(n0) square-free} show that $\alpha(0n)$ and $\alpha(n0)$ are square-free, which is then used to prove Proposition~\ref{alpha(0n0) square-free}.

\begin{lemma}\label{S_1,n suffix}
For all $n \geq 1$, $0203R_3^+R_4^+\cdots R_{n+2}^+R_{n+3}$ is a suffix of $\alpha(n)$.
\end{lemma}

\begin{proof}
We proceed by induction. For $n=1$ we can check directly that $0203R_{3}^+R_{4}$ is a suffix of $\alpha(1)$. For the inductive step, assume that for some $k \geq 1$, we have that $0203R_3^+\cdots R_{k+2}^+R_{k+3}$ is a suffix of $\alpha(k)$ and recall that $$\alpha(k+1) = \alpha(k)^+ \: R_{k+4} \: C\: \alpha(k)^+\: R_{k+4}.$$ Since $\alpha(k)$ ends with $0203R_3^+\cdots R_{k+2}^+R_{k+3}, $ we have that $\alpha(k+1)$ must end with $0203R_3^+\cdots R_{k+2}^+R_{k+3}^+R_{k+4}$, which concludes the proof. 
\end{proof}

\begin{lemma}\label{a, S1 suffix}
For all $n\ge 1$, $R_{n+4}[i:]$ is a suffix of $\alpha(n)^+$ if and only if $i \geq 6$.
\end{lemma}

\begin{proof}
We can write $R_{n+4}$ as 
\begin{align*}
    R_{n+4} &= R_3R_3^+R_4^+\cdots R_{n+2}^+R_{n+3}^+ \\
    &= 01020103R_3^+R_4^+\cdots R_{n+2}^+R_{n+3}^+.
\end{align*}
Also, it follows from Lemma~\ref{S_1,n suffix} that $\alpha(n)^+$ ends with $0203R_3^+\cdots R_{n+2}^+R_{n+3}^+$.  These two words are identical starting with the $03R_3^+$, but not including any letters before.  Therefore, $R_{n+4}[i:]$ is a suffix of $\alpha(n)^+$ if and only if $i \geq 6$.
\end{proof}

\begin{lemma}\label{S1,n occurrences}
For all $n \geq 1$, $\alpha(n)$ ends in $n+3$, and does not contain any letter greater than $n+3$. For $n\geq2$, $\alpha(n)$ contains exactly four occurrences of $n+3$.
\end{lemma}
% alpha(1) has more than 4 occurrences of 4
\begin{proof}
This can be proved using induction. For the base case, it can be checked by direct computation that $\alpha(1)$ and $\alpha(2)$ satisfy the lemma.

For the inductive step, assume that for some $k \geq 2$, $\alpha(k)$ satisfies the lemma. Since $C$ does not contain any occurrence of $k+4$ or higher letters and
\[\alpha(k+1) = \alpha(k)^+R_{k+4} C \alpha(k)^+R_{k+4},\]  
it is clear that our assumption implies that $\alpha(k+1)$ contains exactly four occurrences of $k+4$, no higher letters, and ends with $k+4$.
\end{proof}

\begin{lemma}\label{sigma alphan+ square-free}
For all $n\geq1$, $C \alpha(n)^+$ is square-free. 
\end{lemma}
\begin{proof}
We proceed by induction. We can check computationally that the claim holds for $n=1$. 

For the inductive step assume that  for some $k \geq 1$, $C \alpha(k)^+$ is square-free and suppose that the word $w=C \alpha(k+1)^+$ contains a square. We can think of $w$ as the concatenation of 6 factors:
$$w=w_1w_2w_3w_4w_5w_6:=[C][\alpha(k)^+][ R_{k+4}][C][\alpha(k)^+][R_{k+4}^+].$$
By the inductive hypothesis we have that $w_1w_2=w_4w_5$ is square-free. We divide the rest of the proof in four steps.\\

\textbf{Step 1:} Proving that $w_1w_2w_3$ is square-free. Suppose that $w_1w_2w_3$ contains a square. We know that the square must cross the boundary between $w_2$ and $w_3$, and hence it includes the last letter of $w_2=\alpha(k)^+$, which by Lemma~\ref{S1,n occurrences} is $k+3 + 1 = k+4$.  The only other occurrence of $k+4$ in $w_1w_2w_3$ is the last letter of $w_3=R_{k+4}$, so the square ends in $k+4$. Hence, the second half of the square is the entire $R_{k+4}$, and so $R_{k+4}$ is a suffix of $\alpha(k)^+$, which contradicts Lemma~\ref{a, S1 suffix}. This completes the proof of our first claim.
 \\
 
\textbf{Step 2:} Proving that $w_1w_2w_3w_4$ is square-free. Suppose that $w_1w_2w_3w_4$ contains a square. By the previous step, the square must cross the boundary between $w_3$ and $w_4$, and hence it includes the letter $w_3[-1]=k+4$. Since $k+4$ does not occur in $w_4=C$, and the only other occurrence of $k+4$ in this word is as the last letter of $w_2=\alpha(k)^+$, each half of the square must have length equal to $|w_3| = 2^{k+4}$.
 The largest common prefix between the ruler sequence and $C$ is $C[:5]$. So at most these five letters can appear in each half of the square after the occurrences of $k+4$. On the other hand, Lemma~\ref{a, S1 suffix} implies that at most a suffix of $w_3=R_{k+4}$ with length $2^{k+4}-6$ can appear as a suffix of $w_2=\alpha(k)^+$, and so at most this many letters can appear in each half of the square up to the occurrences of $k+4$.

Therefore, each half of the square contains at most $2^{k+4}-6+5 = 2^{k+4} -1$ letters, which just falls short of the number required.  Therefore no square can exist in $w_1w_2w_3w_4$.\\

\textbf{Step 3:} Proving that $w_1w_2w_3w_4w_5$ is square-free. Suppose that it contains a square. By the previous step the square can not be contained in $w_1w_2w_3w_4$. Also, by the inductive hypothesis the square can not be contained in $w_4w_5$. So the square must contain a nonempty suffix of $w_3=R_{k+4}$, all of $w_4=C$, and a nonempty prefix of $w_5=\alpha(k)^+$. In particular it includes $k+4$, the last letter of $w_3$. Also, note that this word includes only three occurrences of $k+4$ ($w_2[-1]$, $w_3[-1]$ and $w_5[-1]$), hence the square must contain only two of them. 

If the square contains $w_2[-1]$ and $w_3[-1]$, then the second half of the square contains the whole factor $w_4=C$ right after $w_3[-1]$. This implies that $C$ must also appear right after $w_2[-1]$, and so $C$ must be a prefix of $w_3=R_{k+4}$, which is a contradiction. On the other hand, if the square contains $w_3[-1]$ and $w_5[-1]$, then the second half of the square must be $w_4w_5=C\alpha(n)^+$ and the first half must be contained in $w_3=R_{k+4}$. This is impossible since $C$ is not contained in $R_{k+4}$.\\

\textbf{Step 4:} Proving that $w_1w_2w_3w_4w_5w_6$ is square-free. Suppose that it contains a square. The square must contain $w_5[-1]=k+4$ and since $w_6=R_{k+4}^+$ does not contain $k+4$, the occurrence of $k+4$ in the first half of the square must be at $w_3[-1]$. Note that the first half of the square cannot contain the whole $w_4=C$, because this is not a prefix of $w_6=R_{n+4}$. This implies that $w_5=\alpha(n)^+$ is totally contained in the second half of the square. Hence, since $|R_{k+4}|<|\alpha(n)^+|$ we have that $w_3=R_{n+4}$ is a suffix of $w_5=\alpha(n)^+$, which contradicts Lemma~\ref{a, S1 suffix}.

Therefore we conclude that $w_1w_2w_3w_4w_5w_6$ is square-free, which completes the proof of the lemma.
\end{proof}

\begin{lemma}\label{alpha(n) square-free}
$\alpha(n)$ is square-free for all $n \geq 0$.
\end{lemma}
\begin{proof}
We proceed by induction. We can check computationally that $\alpha(0)$ and $\alpha(1)$ are square-free.

For the inductive step assume that $\alpha(k)$ is square-free for some $k \geq 1$. We can think of $w=\alpha(k+1)$ as the concatenation of 5 factors:
\[w=w_1w_2w_3w_4w_5:=[\alpha(k)^+][ R_{k+4}][C][\alpha(k)^+][R_{k+4}].\]
Suppose $w$ contains a square. By Lemma~\ref{sigma alphan+ square-free}, $w^+=\alpha(k+1)^+$ is square-free. So $w[:-1]$ is square-free and the square in $w$ is a suffix, containing $w_5[-1]=k+4$. Also, the square must contain $w_4[-1]=k+4$, since $w_5=R_{k+4}$ is square-free. Now, since $w$ contains exactly four occurrences of $k+4$ we have the following cases:

Case 1: the square contains only $w_4[-1]$ and $w_5[-1]$. In this case the second half of the square would have to be the whole factor $w_5=R_{k+4}$. Hence the first half of the square would have $R_{k+4}$ as a suffix of $w_4=\alpha(n)^+$ which contradicts Lemma~\ref{S1,n occurrences}.

Case 2: the square contains all four occurrences of $k+4$ which are $w_1[-1]$, $w_2[-1]$, $w_4[-1]$, and $w_5[-1]$. In this case $k+4$ must be the final letter in each half of the square, and so the first half of the square would be a suffix of $w_1w_2$. This is not possible since the second half of the square would be $w_3w_4w_5$ which is longer than $w_1w_2$.

Therefore we conclude that $w=\alpha(k+1)$ is square-free, which completes the inductive step.
\end{proof}

\begin{lemma}\label{alpha(0n) square-free}
$\alpha(0n)$ is square-free for all $n\geq1$.
\end{lemma}

\begin{proof}
From Lemma~\ref{alpha(n) square-free} we have that any square in $\alpha(0n)$ must cross into both chunks. We consider two cases based on the length of the square.

Suppose there is a square $yy$ in $\alpha(0n)$ such that $|y|\leq|\alpha(0)|$. Then the square's total length would be at most $|\alpha(0)|\times2=6452$ letters. Computationally, we can see that $\alpha(3)$ and $\alpha(4)$ share their first 13029 letters. Since $\alpha(n)$ is a prefix of $\alpha(n+1)$, all $\alpha(n)$ for $n\geq 3$ have the same first 13029 letters. Thus, $\alpha(0n)$ has the same first $|\alpha(0)|+13029=16255$ letters for all $n \geq 3$.  Checking with a computer, we find that $\alpha(01)$, $\alpha(02)$, and $\alpha(03)$ are square-free, so the first 16255 letters of $\alpha(0n)$ are square-free for all $n$. Since the square must intersect the 0-chunk and has length at most $|\alpha(0)|\times2$, the square must be contained in the first $|\alpha(0)|\times3=9678$ letters. This is a contradiction so no square $|yy|$ with length $|y|\leq|\alpha(0)|$ can occur in $\alpha(0n)$.

Now suppose $|y|>|\alpha(0)|$.  Then the second half of the square is entirely contained in $\alpha(n)$.  However, $\alpha(n)$ is grounded and $\alpha(0)$ ends with 12.  This means the first half of the square can't contain more than the last letter of $\alpha(0)$ so it remains to show that $2\alpha(n)$ is square-free.

Since
\[C[-1]\alpha(n)^+ = 2\alpha(n)^+\]
is a factor of $\alpha(n+1)$, it must be square-free. So we can show that decreasing the last letter by one does not introduce a square.  We know that $2\alpha(n)[1:]$ and $2\alpha(n)[:-1]$ are square-free, so a square would need to be the entire word. However,
\[|2\alpha(n)| = 1 + 2|\alpha(n-1)| + 2|R_{n+3}| + 10,\] which is odd so it is impossible for the entire word to be a square.
\end{proof}

\begin{lemma}\label{alpha(n0) square-free}
$\alpha(n0)$ is square-free for all $n\geq1$. 
\end{lemma}

\begin{proof}
We can check by computer that $\alpha(10)$ is square-free.  Assume $n\geq2$ from now on.

Since $\alpha(0)$ and $\alpha(n)$ are square-free, any square in $\alpha(n0)$ must cross into both chunks and include the $n+3$ at the end of $\alpha(n)$. Since $\max(\alpha(0))=3$, the $n+3$ in each half must come from $\alpha(n)$.  As a result, we know the entire first half of the square is in $\alpha(n)$.  Additionally, $\alpha(0)$ becomes ungrounded at the ninth letter and $\alpha(n)$ is grounded, so the square can't extend past the eighth letter of $\alpha(0)$. Thus this proof simplifies to proving
\[\alpha(n)01020301\]
is square-free.  From Lemma~\ref{S1,n occurrences}, $\alpha(n)$ contains 4 occurrences of $n+3$, so the full square must either contain the last two or all four.

Case 1: The square contains only the third and fourth occurrences of $n+3$, so the square is contained in
\[C\alpha(n-1)^+R_{n+3}01020301 = C\alpha(n-1)^+R_{n+3}C[:8],\]
which is a prefix of $C\alpha(n)^+$, which is square-free by Lemma~\ref{sigma alphan+ square-free}.

Case 2: The square contains all occurrences of $n+3$, so it appears in $$\alpha(n-1)^+R_{n+3}C \alpha(n-1)^+R_{n+3}C[:8].$$  The second and fourth occurrences of $n+3$ are separated by a distance of $$|C| + |\alpha(n-1)^+| + |R_{n+3}| = 10 + |\alpha(n-1)^+| + |R_{n+3}|,$$ so the entire square would need to have twice this length.  However, the whole word has length $$ |\alpha(n-1)^+| + |R_{n+3}| + |C| + |\alpha(n-1)^+| + |R_{n+3}| + |C [:8]| = 18 + 2|\alpha(n-1)^+| +2|R_{n+3}|, $$ which is less than the required length of square.

Since both cases are ruled out, no square can exist in $\alpha(n0)$. 
\end{proof}

We can now prove that $\alpha(0n0)$ is square-free:

\begin{proof}[Proof of Proposition~\ref{alpha(0n0) square-free}]
We can prove via computation that $\alpha(010)$ and $\alpha(020)$ are square-free.  Assume $n\geq3$ from here. 

Using Lemmas~\ref{alpha(0n) square-free} and \ref{alpha(n0) square-free}, it follows that if a square exists in $\alpha(0n0)$, then it includes the entire $n$-chunk as well as a part of each $0$-chunk. We can check using a computer that $|\alpha(0)| = 3226$ and $|\alpha(3)|=13030$. For any $p>q>1$, $|\alpha(p)|>|\alpha(q)|$.  Therefore, $|\alpha(n)|\geq |\alpha(3)|>|\alpha(0)|$.

Case 1: Suppose the boundary between the halves of the square appears either in a $0$-chunk or between chunks.  Then the half of the square entirely in a $0$-chunk would have a length at most 3226.  The other half would also have the same length, but we know this half needs to contain the entire $n$-chunk, which has length greater than 3226.  Thus this case is impossible.

Case 2: Now suppose the boundary lies within the $n$-chunk. Let $x$ be the nonempty suffix of the first 0-chunk contained in the square, and $z$ be the nonempty prefix of the last 0-chunk contained in the square. The occurrence of $x$ in the second half of the square begins in $\alpha(n)$. If it extends into the last 0-chunk, then the whole square has length less than
\[|xzxz|\leq4\times|\alpha(0)|=12904<13030=|\alpha(3)|\leq|\alpha(n)|.\]
This is not possible since the square contains all of $\alpha(n)$. Thus, there are nonempty words $y$ such that the square can be written as
\[x][yzxy][z,\]
where $yzxy = \alpha(n)$. Since all of $\alpha(n)$ is grounded, $x$ and $z$ are both grounded.  We can check by computer that the longest grounded prefix of $\alpha(0)$ is 01020301, and its longest grounded suffix is 2.  It follows that $x = 2$ and $z$ is a prefix of 01020301.  This means that $|zx| \leq 9$ and $zx$ must appear in the exact center of $\alpha(n)$.  The middle 10 letters of $\alpha(n)$ are $C = 0102030102$, so $zx$ is located at the center of $C$, meaning that it must be grounded. Thus, the only possible values for $zx$ are 02, 0102, 010202, and 01020302.  Clearly none of these appear at the center of $C$, so this square cannot exist.
\end{proof}

We can now use the above results to prove that $\alpha$ is square-free over grounded words.

\begin{proof}[Proof of Theorem~\ref{alpha square-free over grounded words}]
Suppose $w$ is a square-free grounded word and $\alpha(w)$ contains a square $yy$. We will first show that $E$ must be a factor of $y$, and then use that to show that both halves of the square have the same chunk decomposition. We then show that each half of the square contains a whole chunk, and that if the square contains any partial chunks, then the initial partial chunks of the two halves of the square come from the same letter. It is clear from its definition that $\alpha$ is letter-injective, and we know from Lemma~\ref{alpha(n) square-free} that $\alpha(n)$ is square-free for all letters $n$. Hence, Theorem~\ref{chunk-identifiable condition for square-freeness} will imply that $w$ contains a square, which is a contradiction.
	
The square $yy$ contains a whole $0$-chunk, since otherwise it would be a factor of $\alpha(0n0)$, contradicting Proposition~\ref{alpha(0n0) square-free}. Since $\alpha(0)=EFE$, there are at least two whole occurrences of $E$ in $yy$. At least one of these occurrences must be completely contained in one half of the square. Thus, $E$ is a factor of $y$.
	
Let $[\alpha(l)],\ 0\leq l$ be any whole chunk in either half of the square. We will show that the corresponding occurrence of $\alpha(l)$ in the other half is also a whole $l$-chunk. If $l=0$, then by Corollary~\ref{alpha(0) is a 0-chunk}, $\alpha(l)$ must be a whole $0$-chunk in both halves. If $l>0$, then since $w$ is grounded and $E$ is a factor of $y$, there must be a whole occurrence of $E$ adjacent to this chunk and entirely contained in this half of the square. Thus, either $E\alpha(l)$ or $\alpha(l)E$ is a factor of $y$. Then by Corollary~\ref{Ealpha(l) and alpha(l)E indicate l-chunks}, $\alpha(l)$ is a whole $l$-chunk in both halves of the square, so both halves have the same chunk decomposition.
    
Suppose neither half of the square contains a whole chunk. Then $yy$ cannot span over more than three chunks. Since $\alpha(0n0)$ is square-free by Proposition~\ref{alpha(0n0) square-free}, $yy$ must be a proper factor of $\alpha(n0k)$. Then $yy$ overlaps all three chunks because $\alpha(n0)$ and $\alpha(0k)$ are square-free. By Lemma~\ref{E is prefix or suffix of alpha(0)}, there are exactly two occurrences of $E$ in $\alpha(n0k)$ and each must be in a different half of the square since $E$ is a factor of $y$. But the letter before the $E$ in the first half is the last letter of $\alpha(n)$, which is $n+3$, and the letter before the $E$ in the second half is the last letter of $F$, which is $2$. We cannot have $E$ as a prefix of $y$, or else the square would not overlap the first of the three chunks. This is a contradiction so one half of the square must contain a whole chunk. Since both halves have the same chunk decomposition, both halves contain a whole chunk.
    
Suppose either half of the square contains a partial chunk. Then since the halves have the same chunk decomposition, they must both begin with a partial chunk. We will show that the halves share their initial partial chunk. The initial partial chunks end with the same letter, so by Lemma~\ref{alpha shared prefix and last letter}, they must be equal chunks or one of them must be a $0$-chunk. But since the halves of the square have the same chunk decomposition and contain a whole chunk, their first whole chunks are equal. So the final partial chunks are either both $0$-chunks or equal nonzero chunks.
		
This verifies the conditions of Theorem~\ref{chunk-identifiable condition for square-freeness}  which implies that $w$ contains a square, a contradiction.
\end{proof}

%%%%%%%%%%%%%%%%%%%%%%%%%%%%%%%%
\subsubsection{Conditions~\ref{L-commuting over grounded condition 0 to 01}, \ref{L-commuting over grounded condition 0n to 0n0}, and \ref{L-commuting over grounded condition 0n-plus to 0n+1}}
In this section we prove that $\alpha$ satisfies the remaining conditions of Theorem~\ref{L-commuting over grounded words template}. Condition~\ref{L-commuting over grounded condition 0 to 01} ($\alpha(0)$ generates $\alpha(01)$) can be verified via direct computation. In the following result we prove that $\alpha$ satisfies Condition~\ref{L-commuting over grounded condition 0n to 0n0}.

\begin{theorem}%\label{alpha condition 2}
For all $n > 0$, $\alpha(0n)$ generates $\alpha(0n0)$.
\end{theorem}

\begin{proof}
The case $n=1$ can be verified via direct computation, so we assume $n\ge 2$. By Theorem~\ref{alpha(0n0) square-free} $\alpha(0n0)$ is square-free, hence since the word $G$ generates $\alpha(0)$, it is enough to show that $G$ is irreducible in $\alpha(0n)G$. Indeed, consider $$\alpha(0n) \, G = \alpha(0n) \, 01020\underline{3}01\underline{2}.$$ 
It is clear that the only letters in $G$ that could potentially be reduced are the underlined ones. The 3 could only be reduced to 1, in which case note that
\begin{align*}
\alpha(0n) \, 010201 &= \cdots \alpha(n) \, 010201 \\
&= \cdots \alpha(n-1)^+R_{n+3}C\alpha(n-1)^+R_{n+3} \, 010201 \\
&= \cdots \alpha(n-1)^+R_{n+3}010201\\
&=\cdots R_{n+3}[6:]R_{n+3}010201 &&\textrm{(by Lemma~\ref{a, S1 suffix})}\\
&=\cdots R_{n+3}[6:]010201R_{n+3}[6:]010201,
\end{align*}
which contains a square, so the 3 in $G$ is irreducible.

Now, the last 2 in $G$ could only be reduced to 0, and in this case note that
\begin{align*}
\alpha(0n) \, 010203010 &= \cdots 2 \, \alpha(n) \, 010203010 \\
&= \cdots 2 \, \alpha(n-1)^+R_{n+3}C\alpha(n-1)^+R_{n+3} \, 010203010 \\
&= \cdots 2 \, \alpha(n-1)^+R_{n+3}0102030102\alpha(n-1)^+R_{n+3} \, 010203010,
\end{align*}
which contains a square, so the last letter of $G$ is irreducible.
\end{proof}

Before proving that $\alpha$ satisfies Condition~\ref{L-commuting over grounded condition 0n-plus to 0n+1} of Theorem~\ref{L-commuting over grounded words template} we need to establish the following lemmas.
\begin{lemma}\label{sigma irreducible}
$C$ is irreducible in $\alpha(n)^+R_{n+4}C$ for all $n>0$.
\end{lemma}

\begin{proof}
Recall that $C = 0102030102$. Clearly, the only letter that is reducible within $C$ is the 3, which could only be made a 1.  In this case we would have $\alpha(n)^+ R_{n+4} 010201$. Since $n\ge 1$, Lemma~\ref{a, S1 suffix} says that $R_{n+4}[6:]$ is a suffix of $\alpha(n)^+$. Also, $R_{n+4}[:6]=010201$ for all $n$. Therefore
\[\alpha(n)^+R_{n+4}010201=\cdots R_{n+4}[6:]010201R_{n+4}[6:]010201\]
which contains a square. Hence $C$ is irreducible.
\end{proof}

\begin{lemma}\label{sigma generates alpha(k)+}
$C$ generates $C\alpha(n)^+$ for all $n>0$.
\end{lemma}

\begin{proof}
We proceed by induction. We can check in the case $n=1$, that $C\alpha(1)^+$ is a prefix of $L(C)$ by direct computation.

For the inductive step, assume that $C$ generates $C\alpha(k)^+$ for some $k \geq 1$.  First note that \[C \alpha(k+1)^+=C\alpha(k)^+R_{k+4}C\alpha(k)^+R_{k+4}^+\]
is square-free, since it is a factor of $\alpha(k+2)$. Then we just need to show that $\alpha(k+1)^+$ is irreducible. 

From the inductive hypothesis $\alpha(k)^+$ is irreducible after $C$ and since $R_{k+4}$ is a prefix of the ruler sequence, it is also irreducible. Now, Lemma~\ref{sigma irreducible} implies that $C$ is irreducible after $C\alpha(k)^+R_{k+4}$ and from the inductive hypothesis again we conclude that that $\alpha(k)^+$ is irreducible after $C\alpha(k)^+R_{k+4}C$. Finally, the last letter in $R_{k+4}^+$ cannot be reduced by 1 because it would create the square $(C\alpha(k)^+R_{k+4})^2$, and cannot be reduced by more than 1 because $R_{k+4}$ is irreducible. Therefore $\alpha(k)^+R_{k+4}C\alpha(k)^+R_{k+4}^+=\alpha(k+1)^+$ is irreducible in $C \alpha(k+1)^+$, which concludes the proof.
\end{proof}

We note that Lemma~\ref{sigma generates alpha(k)+} immediately describes the structure of $L(012)$.
\begin{corollary}\label{L(012)}
$L(012)=01201\lim_{n\to\infty}\rho^{-1}(\alpha(n))$.
\end{corollary}\begin{proof}
First note that $\rho^{-1}$ is well defined on $\alpha(n)$ since it is even-grounded. Since $010203$ generates $C$, Lemma~\ref{sigma generates alpha(k)+} implies that $L(\rho(012))=L(010203)=L(C)=C\lim_{n\to\infty}\alpha(n)$. Since $\rho$ is $L$-commuting over square-free words by Theorem~\ref{rho is L-commuting}, $L(\rho(012))=\rho(L(012))$. Thus, $L(012)=\rho^{-1}(C)\rho^{-1}(lim_{n\to\infty}\alpha(n))=01201\lim_{n\to\infty}\rho^{-1}(\alpha(n))$.
\end{proof}

Finally, we prove that $\alpha$ satisfies Condition~\ref{L-commuting over grounded condition 0n-plus to 0n+1}.

\begin{theorem}%\label{alpha condition 4}
$\alpha(0n)^+$ generates $\alpha(0(n+1))$ for all $n>0$.
\end{theorem}

\begin{proof}
From Lemma~\ref{alpha(0n) square-free} we know that $\alpha(0(n+1))$ is square-free, so it is enough to show that $\alpha(n)^+$ generates $\alpha(n+1)$. To show this, recall that \[\alpha(n+1)=\alpha(n)^+R_{n+4}C\alpha(n)^+R_{n+4}.\]
Consider that $\alpha(n)^+$ generates $\alpha(n)^+ \, R_{n+4} \, C$, because $R_{n+4}$ is a prefix of the ruler sequence and $C$ is irreducible by Lemma~\ref{sigma irreducible}. Similarly, Lemma~\ref{sigma generates alpha(k)+} implies that $\alpha(n)^+ \, R_{n+4}$ is irreducible after $C$. Therefore $\alpha(n)^+$ generates $\alpha(n+1)$.
\end{proof}

%%%%%%%%%%%%%%%%%%%%%%%%%%%%%%%%
\subsubsection{Conclusion}

We have proved that the morphism $\alpha$ satisfies all the requirements of Theorem~\ref{L-commuting over grounded words template}, hence we have the following result.

\begin{theorem}\label{alpha lcommuting}
$\alpha$ is $L$-commuting over $\Sigma$, the set of all nonempty even-grounded square-free words.
\end{theorem}

\begin{corollary}\label{L(C)=alpha(L(E))}
$L(G) = L(\alpha(0)) = \alpha(L(\varepsilon))$.
\end{corollary}

\begin{proof}
Recall that $G$ generates  $\alpha(0)$, so $L(G) = L(\alpha(0))$. The other equality follows directly from Theorem~\ref{alpha lcommuting}, since $0\in\Sigma$.
\end{proof}

%%%%%%%%%%%%%%%%%%%%%%%%%%%%%%%%%%%%%%%%%%%%%%%%%%%%%%%%%%%%%%%%%%%%%%%%%%%%
\subsection{Structure of $L(1)$ and $L(n)$ for $n\ge 3$}\label{L(1) section}

The following result will reduce the task of proving the square-freeness of a word formed by a finite prefix followed by $\alpha(L(\varepsilon))$ to a finite computation.

\begin{lemma}\label{lemma for adding a prefix to alpha of ruler}
Let $w$ be a finite square-free word. If $w\alpha(L(\varepsilon))$  contains a square, then that square contains no letter greater than $\max(w \alpha(0))$.
\end{lemma}

\begin{proof}
Suppose toward a contradiction that there is a square $yy$ with a letter greater than $\max(w \alpha(0))$. Since $w$ and $\alpha(L(\varepsilon))$ are square-free, $yy$ must cross the boundary between these two factors. Choose $n$ to be some letter such that $\max(\alpha(n))$ is greater than any letter in $y$. Then $yy$ must be contained in $w \, \alpha(R_{n}[:-1]\,n) = w \, \alpha(R_{n})$.

Let $l:=\max(y)>\max(w \alpha(0))$ and choose some occurrence of $l$ in the first half of the square. Since $l$ is neither contained in $w$ nor $\alpha(0)$, and since $R_n$ is even-grounded, this occurrence of $l$ is from an $i$-chunk which is after a $0$-chunk. Also, this $0$-chunk must be totally contained within the first half of the square, since the square involves $w$.

Let $s$ be the suffix of the first half of the square starting right after $w$. By our previous reasoning, $s$ contains the whole first 0-chunk. We claim that the occurrences of $s$ at the end of each half of the square $yy$ have the same chunk decomposition. Indeed, let $s_1$ and $s_2$ be the occurrences of $s$ in the first and second half of the square respectively. Consider the first whole occurrence of $\alpha(0)$ in $s_1$ and $s_2$, Lemma~\ref{alpha(0) is a 0-chunk} implies that these occurrences of $\alpha(0)$ are indeed $0$-chunks.

Now recall that by Lemma~\ref{alpha shared prefix and last letter}, for $n>k>0$, we have that $\alpha(n)$ and $\alpha(k)$ end with different letters and neither is a prefix of the other. Hence, it is not hard to see that if  for some $n>0$, there is a whole $n$-chunk in $s_1$ either right before or immediately after this $0$-chunk, then $s_2$ must also have this $n$-chunk in the same position. Inductively we have that $s_1$ and $s_2$ have the same chunk decomposition.

The first half of the square has no initial partial chunk. If the occurrence of $l$ in the first half of the square is in the final partial chunk, then that chunk ends with a suffix of $w$ (which is the initial partial chunk of the second half). But by Lemma~\ref{S1,n occurrences}, the last letter of the chunk is its largest letter which is at least $l$. This is a contradiction since $l>\max(w)$, so $l$ occurs in a whole $i$ chunk. Since $s_1$ and $s_2$ have the same chunk decomposition, $l$ occurs in an $i$-chunk in both halves of the square.

From our knowledge of the ruler sequence, for any two occurrences of $i$ within $R_n$ there exists an $i+1$ between them, and so $\alpha(i+1)[-1]$ is contained within $yy$. Finally, from Lemma~\ref{S1,n occurrences} we have that $\alpha(i+1)[-1] > \alpha(i)[-1]\geq l$, which contradicts our choice of $l$.
\end{proof}

\begin{remark}\label{adding a prefix to rho of alpha of ruler}
We can see from the properties of the ruler sequence that Lemma~\ref{alpha(0) is a 0-chunk}, Lemma~\ref{S1,n occurrences}, and Lemma~\ref{alpha shared prefix and last letter} apply analogously to the morphism $\rho\circ\alpha$. Then the proof of Lemma~\ref{lemma for adding a prefix to alpha of ruler} can be easily adapted to show that if $w$ is square-free, then any square in $w\rho(\alpha(L(\varepsilon)))$ contains no letter greater than $\max(w\rho(\alpha(0)))$. This will be used in Lemma~\ref{Arho(alpha(L(0))) is square-free} to prove that $A\rho(\alpha(L(\varepsilon)))$ is square-free.
\end{remark}

Now we can prove Theorem~\ref{structure of L(1)}.

\begin{theoremL(1)}
Let $Y_1$ be the $5177$-letter prefix of $L(1)$.
Then $L(1) = Y_1 \, \alpha(L(\varepsilon))$.
\end{theoremL(1)}

\begin{proof}
We first show that $Y_1\alpha(L(\varepsilon))$ is square-free. Indeed, suppose that $Y_1\alpha(L(\varepsilon))$ contains a square. Since $Y_1$ and $\alpha(L(\varepsilon))$ are square-free, the square must start in $Y_1$. We can verify by computation that $Y_1\alpha(R_2)$ is square-free, so the square must end after $\alpha(R_2)$. Hence it contains $\alpha(2)[-1]=5$. But by Lemma~\ref{lemma for adding a prefix to alpha of ruler}, the square cannot contain any letter larger than $\max(Y_1\alpha(0))=4$. This is a contradiction so $Y_1\alpha(L(\varepsilon))$ is square-free.

We can check by direct computation that $L(1) = L(Y_1 G)$. Then using Remark~\ref{Appending to front of lex-least words} with $p=G$, $w=L(G)$, and $u=Y_1$, and Corollary~\ref{L(C)=alpha(L(E))}, we obtain that $$L(1)=L(Y_1 G) = Y_1 \, L(G) = Y_1 \alpha(L(\varepsilon)).$$
\end{proof}

The structure of $L(n)$ for $n \geq 3$ is similar to that of $L(1)$, although the prefix is different and the morphism $\alpha$ is replaced with the composition $\rho \circ \alpha$. 
From Theorem~\ref{tau prefix theorem}, we know that $L(n)$ has prefix $nT(n)$ which has length exponential in $n$. This is followed by $A$, a constant word of length $13747$ which can be easily found computationally. It is noteworthy that $A$ has prefix $\psi_2(0)^+$.

\begin{theoremL(n)}
For all $n \geq 3$, $L(n) = Y_n\, \rho(\alpha(L(\varepsilon)))$, where $Y_n=n \, T(n) \, A$.
\end{theoremL(n)}

In order to prove that $n$ generates $nT(n)A\rho(\alpha(\varepsilon))$, we need first to show that it is square-free and then show that $A\rho(\alpha(\varepsilon))$ is irreducible. We begin with some lemmas used to prove the square-free condition.

%%%%%%%%%%%%%%%%%%%%%%%%%%%%%%%%

\begin{lemma}\label{ntau(n)A is square-free}
For $n\geq3$, $nT(n)A$ is square-free.
\end{lemma}
\begin{proof}
From Theorem~\ref{tau prefix theorem} we know that $nT(n)$ is square-free and we can verify that $A$ is square-free computationally, so any square would have to overlap both factors. We can also computationally check the cases $n=3,4,5$, so assume $n\geq6$ and suppose that $nT(n)A$ contains a square $yy$.

From Remark~\ref{structure of P2 and tau(n) has suffix tau(3)} we have that $T(n)$ has suffix $T(6)$. We can computationally verify that $T(6)A$ is square-free, so the square contains $T(6)$. Since $\max(T(6)A)=6$, then let $k$ be the largest letter in the square, we have that $k\geq6$. Also, $\max(A)=5$, so all occurrences of $k$ are in $nT(n)$. Since both halves contain at least one letter $k$, then the center of the square lies in $nT(n)$.

Recall that $A$ begins with $\psi_2(0)^+$ which never occurs in $nT(n)=\psi_2((n-2)P_0(n-2))$ by Lemma~\ref{psi_2(0) is a 0-chunk}. Since the first half is contained in $nT(n)$, $y$ cannot contain $\psi_2(0)^+$. The second half of the square starts in $nT(n)$, but it cannot contain all of $\psi_2(0)^+$ at the beginning of $A$. Therefore, the square is a factor of
\[nT(n)\psi_2(0)[:-1]=\psi_2((n-2)R_{n-2}[:-1](n-2)R_{n-2}[:-1])[:-1].\]
We first consider that $R_{n-2}[:-1](n-2)R_{n-2}[:-1]=R_{n-1}[:-1]$ is square-free, so by Proposition~\ref{psi_2 square-free over grounded}, $\psi_2(R_{n-2}R_{n-2}[:-1])$ is square-free meaning that the square intersects the first chunk, $\psi_2(n-2)$.

This means that the square contains all of the middle occurrence of ${\psi_2(n-2)}$ which contains four occurrences of $n$. Since $yy$ contains $n$ and ${\max(nT(n))}=n$ we have $n=k$. Also, there are exactly 8 occurrences of $n$ in $nT(n)A$, 4 in each occurrence of $\psi_2(n-2)$. The square $yy$ must include the last 4 occurrences of $n$ and either none of the earlier ones, just the last 6 $n$'s, or all 8 $n$'s. Recall that
\begin{multline*}
   nT(n)\psi_2(0)[:-1]=\\\underline{\psi_2(n-2)}\ \psi_2(R_{n-2}[:-2])\ \psi_2(0)\ \underline{\psi_2(n-2)}\ \psi_2(R_{n-2}[:-2])\ \psi_2(0)[:-1] 
\end{multline*}
and from Equation~\eqref{psi_2(n) decomposition} in Section~\ref{section P prefixes},
\[\psi_2(n-2)=\underline{n}\ \underline{R_n}\ R_n[:-2]\ \psi_1(R_{n-1}[:-1])\ \underline{n}\ \underline{R_n}\ R_n[:-2]\ \psi_1(R_{n-1}[:-2]),\]
which shows the locations of the 4 occurrences of $n$ in $\psi_2(n-2)$.

If $yy$ contains only the last 4 occurrences of $n$, then $|y|=|nR_nR_n[:-2]\psi_1(R_{n-1}[:-1])|$. But $yy$ contains $\psi_2(0)\psi_2(n-2)$ which has length more that twice the length of $nR_nR_n[:-2]\psi_1(R_{n-1}[:-1])$. This is a contradiction.

If $yy$ contains only the last 6 occurrences of $n$, we consider the first and second $n$ in each half. The first two $n$'s in the first half occur together as $nR_n$. But the first two $n$'s in the second half occur in $\underline{R_n}\ R_n[:-2]\ \psi_1(R_{n-1}[:-1])\ \underline{n}$. Clearly, $nR_n[:-2]\psi_1(R_{n-1}[:-1])n\neq nR_n$, so this is a contradiction.

If $yy$ contains all 8 occurrences of $n$, then it starts at the first letter of $nT(n)$. Then $y=\psi_2(n-2)\psi_2(R_{n-2}[:-2])\psi_2(0)$, but then $yy=nT(n)\psi_2(0)$ is not a factor of $nT(n)\psi_2(0)[:-1]$. So this is also a contradiction.
\end{proof}

\begin{lemma}\label{Arho(alpha(L(0))) is square-free}
$A\rho(\alpha(L(\varepsilon)))$ is square-free.
\end{lemma}
\begin{proof}
Suppose $A\rho(\alpha(L(\varepsilon)))$ contains a square. Since $A$ is square-free, we can use Remark~\ref{adding a prefix to rho of alpha of ruler} to see that the square contains no letter greater than $\max(A\rho(\alpha(0)))=5$. Since $\rho(\alpha(L(\varepsilon)))$ is square-free, the square overlaps $A$. For all letters $n\ge0$, $\max(\rho(\alpha(n)))=n+4$, so the square would need to be contained in $A\rho(\alpha(R_2))$ which is a prefix of $A\rho(\alpha(L(\varepsilon)))$ that contains 6. We can computationally verify that $A\rho(\alpha(R_2))$ is square-free which is a contradiction.
\end{proof}

\begin{theorem}\label{L(n>2) square-free condition}
For all $n\geq3$, $nT(n)A\rho(\alpha(L(\varepsilon)))$ is square-free.
\end{theorem}

\begin{proof}
Suppose that there is a square $yy$ in $nT(n)A\rho(\alpha(L(\varepsilon)))$. We have from Lemmas~\ref{ntau(n)A is square-free} and~\ref{Arho(alpha(L(0))) is square-free} that $nT(n)A$ and $A\rho(\alpha(L(\varepsilon)))$ are both square-free. So the square $yy$ must contain all of $A$ and overlap some nonempty suffix of $nT(n)$ and some nonempty prefix of $\rho(\alpha(L(\varepsilon)))$.

Consider the prefix $p := {A[:254]}$ and the suffix $s := {A[-88:]}$, and define $w$ such that $A=pws$. Since $A$ is totally contained in $yy$, then at least one of $p$ or $s$ must be totally contained in $y$. We will show that $p$ and $s$ each occur exactly once in $nT(n)A\rho(\alpha(L(\varepsilon)))=nT(n)pws\rho(\alpha(L(\varepsilon)))$, which leads to a contradiction, since at least one of $p$ or $s$ must appear in both halves of the square.

First we show that $p$ appears exactly once in $nT(n)A\rho(\alpha(L(\varepsilon)))$. We can verify computationally that $p$ occurs exactly once in $A$. Also, since $p$ begins with $\psi_2(0)^+$ which  by Lemma\ref{psi_2(0) is a 0-chunk} never occurs in $nT(n)$, then $p$ cannot occur in $nT(n)$. Moreover, if $p$ occurred over the boundary between $nT(n)$ and $A$, then since $p$ is a prefix of $A$, there would be a square in $nT(n)A$ which is not true according to Lemma~\ref{ntau(n)A is square-free}. Finally, since $p$ ends with 12 and $\rho(\alpha(L(\varepsilon)))$ is even-grounded we conclude that $p$ occurs neither in $\rho(\alpha(L(\varepsilon)))$ nor on the boundary between $A$ and $\rho(\alpha(L(\varepsilon)))$.

Secondly, we show that $s$ appears exactly once in $nT(n)A\rho(\alpha(L(\varepsilon)))$. It can be checked that $s$ occurs exactly once in $A$. To show that $s$ is not contained in $nT(n)=\psi_2((n-2)P_0(n-2))$ we use the following properties of $s$, which can be verified computationally: $s$ is even-grounded, $\psi_2(0)$ does not contain $s$, $\max(s)=4$ and $s$ contains 7 occurrences of 4. Since $s$ is even-grounded and for all $n\ge 0$, $\psi_2(n)$ begins and ends with nonzero letters, $s$ cannot lie over a $\psi_2$ chunk boundary. So if $s$ occurs in $nT(n)$, it is within $\psi_2(k)$ for some $k>0$. We can computationally verify that $s$ does not occur in $\psi_2(1)$ or $\psi_2(2)$, so assume $k>2$. Consider that $\psi_2(k)=\psi_1((k+1)P_0(k+1))=\psi_1((k+1)R_{k+2}[:-2])$ and that $\psi_1(\ell)$ contains no 4's when $\ell<3$, two 4's and no 5's when $\ell=3$, and contains 5's when $\ell>3$. So if $s$ is contained in $\psi_2(k)$, it must contain at least two whole occurrences of $\psi_1(3)$, and no whole occurrence of $\psi_1(4)$. But since $k>2$, any two occurrences of 3 in $(k+1)R_{k+2}[:-2]$ have an occurrence of 4 between them. So $s$ cannot be contained in $nT(n)$.

 Note that $A$ begins with 2 and recall from Remark~\ref{structure of P2 and tau(n) has suffix tau(3)} that $T(n)$ has suffix $T(3)$ which ends with a 1. Hence, since  $s$ is grounded, $s$ cannot lie over the boundary between $nT(n)$ and $A$. Also, if $s$ occurred over the boundary between $A$ and $\rho(\alpha(L(\varepsilon)))$, then since $s$ is a suffix of $A$, there would be a square in $A\rho(\alpha(L(\varepsilon)))$, which is not true according to Lemma~\ref{Arho(alpha(L(0))) is square-free}. 

Finally we show that $s$ cannot occur in $\rho(\alpha(L(\varepsilon)))$. Since $s$ is even-grounded and has even length, $\rho^{-1}(s)$ is well-defined. Hence, it is enough to show that $\rho^{-1}(s)$ does not occur in $\alpha(L(\varepsilon))$. The first two letters of $\rho^{-1}(s)$ are 13. For $n\geq0$, $\alpha(n)$ begins with zero, so 13 cannot occur over a chunk boundary in $\alpha(L(\varepsilon))$. Also, for $n\geq1$, $\alpha(n)$ is grounded so it does not contain 13. We can verify directly that $\rho^{-1}(s)$ does not occur in $\alpha(0)$, so $s$ does not occur in $\rho(\alpha(L(\varepsilon)))$.
\end{proof}

\begin{proposition}\label{L(n>2) prefix condition}
For $n\geq3, L(n)$ has prefix $nT(n)A\rho(G)$.
\end{proposition}

\begin{proof}
We know from Theorem~\ref{tau prefix theorem} that $L(n)$ has prefix $nT(n)$. First we show that $L(n)$ has prefix $n \, T(n) \, \psi_2(0)^+$. Indeed, using Remark~\ref{structure of P2 and tau(n) has suffix tau(3)}, $T(n)$ has suffix $T(3)$ which has suffix $\psi_2(1)$, which according to Lemma~\ref{psi_2(n) generates psi_2(n0)} generates $\psi_2(10)$. This means that $\psi_2(0)$ at the end of $n \, T(n) \, \psi_2(0)$ is irreducible. However, $\psi_2(0)$ also introduces a square:
\begin{align*}
   n \, T(n)\psi_2(0)&=n\,P_0(n)\,P_1(n)\,\,\psi_2(R_{n-2}R_{n-2}[:-2])\,\psi_2(0)\\
   &=\psi_2(n-2)\,\,\,\,\,\,\,\,\,\,\,\,\psi_2(R_{n-2}[:-1])\,\,\psi_2(n-2)\,\psi_2(R_{n-2}[:-2])\,\psi_2(0)\\
   &=\psi_2(n-2)\,\,\,\,\,\,\,\,\,\,\,\,\psi_2(R_{n-2}[:-1])\,\,\psi_2(n-2)\,\psi_2(R_{n-2}[:-1])\\
   &= \psi_2((n-2) \, R_{n-2}[:-1])^2.
\end{align*}
Therefore $n \, T(n) \, \psi_2(0)^+$ is irreducible. It is also square-free by Proposition~\ref{L(n>2) square-free condition}, since $A$ has prefix $\psi_2(0)^+$.

Now, from Remark~\ref{structure of P2 and tau(n) has suffix tau(3)} we know that $n \, T(n) \, \psi_2(0)^+$ has suffix $T(3) \, \psi_2(0)^{+}$. A computer can then verify that this generates $T(3)A\rho(G)$. Since $G$ is a prefix of $\alpha(0)$ we obtain from Proposition~\ref{L(n>2) square-free condition} that $n \, T(n) \, A \, \rho(G)$ is square-free, which concludes the proof.
\end{proof}

%%%%%%%%%%%%%%%%%%%%%%%%%%%%%%%%
We can now prove that $L(n)=Y_n\rho(\alpha(L(\varepsilon)))$.
\begin{proof}[Proof of Theorem~\ref{L(n) for large n}]
From Proposition~\ref{L(n>2) prefix condition} we get $L(n) = L(n \, T(n) \, A \, \rho(G))=L(Y_n\rho(G))$. Also, Theorem~\ref{rho is L-commuting} implies that $L(\rho(G))=\rho(L(G))$, so by Corollary~\ref{L(C)=alpha(L(E))} we have that
\[Y_nL(\rho(G))=Y_n\rho(L(G))= Y_n \rho(\alpha(L(\varepsilon))),\]
which is square-free by Proposition~\ref{L(n>2) square-free condition}. Hence, Theorem~\ref{L(n) for large n} follows from Remark~\ref{Appending to front of lex-least words} with $p=\rho(G)$, $u=Y_n$ and $w=L(\rho(G))$.
\end{proof}

%%%%%%%%%%%%%%%%%%%%%%%%%%%%%%%%
\subsection{The structure of $L(2)$}\label{L(2) section}

In this section, we briefly describe a conjectured structure for $L(2)$ that is similar to the structures of $L(n)$ in the previous sections.

First note that $R_n$ can be written recursively as
\begin{align*}
  R_1 &= 01,\\
  R_n &= R_{n-1} R_{n-1}^+.
\end{align*}
Define
\begin{align*}
  b_2 &= 0102012021012, \\
  b_n &= b_{n-1}b_{n-1}^+R_{n-1}R_{n-1}^+ = b_{n-1}b_{n-1}^+R_{n}.
\end{align*}
Also, let $c_3$ be the $261$-letter word:
\begin{align*}
  c_3 = {}
  & 0102012021012010201202102010210120102012021012010201301020103 \\
  & 0102012021012010201202101301020103010201202101201020120230102 \\
  & 0103010201202101201020120301020103010203010302010203010201030 \\
  & 1020301030201202101201020120210120230102010301020120210120102 \\
  & 01202101301020103, \\
  c_n = {}
  & c_{n-1} c_{n-1}^+ R_{n-1} R_{n-1}^+ b_{n-1} b_{n-1}^+ R_{n-1} R_{n-1}^+ \\
  = {} & c_{n-1} c_{n-1}^+ R_{n} b_{n}.
\end{align*}
Notice that for all $n$, $c_n$ has $c_{n-1}$ as a prefix. Thus, $c_n$ has all previous $c_k$ ($3\le k<n$) as prefixes and we have
\begin{align*}
\lim_{n \to \infty}c_n&=c_{3}\ c_{3}^+R_{4}b_{4}\ c_{4}^+R_{5}b_{5}\ c_{5}^+R_{6}b_{6}\cdots\\
&=c_3c_3^+\ R_4b_4\ c_3c_3^+\ R_4b_4^+R_5b_5\ c_4c_4^+\ R_5b_5^+R_6b_6\cdots\\
&=\underline{c_3c_3^+}\ R_4b_4\ \underline{c_3c_3^+}\ R_4b_4^+R_5b_5\ \underline{c_3c_3^+}\ R_4b_4\ \underline{c_3c_3^+}\ R_4b_4^+R_5b_5^+R_6b_6\cdots,
\end{align*}
which gives rise to the following morphism:

\begin{definition*}%\label{def_of_gamma}
For all $n \geq 0$, $\gamma(n)$ is the morphism defined by
\begin{align*}
\gamma(0) &= c_3 c_3^+ \\
\gamma(n) &= R_4 b_4^+ \; R_5 b_5^+ \cdots R_{n+2} b_{n+2}^+ \; R_{n+3} b_{n+3}.
\end{align*}
\end{definition*}
From the structure of $c_n$ and $\gamma$, we can see that
\[\lim_{n \to \infty}c_n=\gamma(L(\varepsilon)).\]

\begin{conjectureL(2)}
$\displaystyle{L(2) = 2 \lim_{n \to \infty} c_n = 2 \gamma(L(\varepsilon))}$.
\end{conjectureL(2)}

%%%%%%%%%%%%%%%%%%%%%%%%%%%%%%%%
\section{Extending from known words}\label{section: extending}

In this section we give two results establishing conditions for when $L(uv)=uL(v)$ for words $u$ and $v$. If $uv$ is square-free, then so is $L(uv)$. Thus, by Remark~\ref{Appending to front of lex-least words}, $uL(v)$ being square-free is a necessary and sufficient condition for $L(uv)=uL(v)$ when $uv$ is square-free. The following result lets us use our knowledge of $L(n)$ to show that we can omit the square-free condition when $u$ and $v$ are letters $\geq3$, and Theorem~\ref{sufficient check} demonstrates a test for the case when $uv$ is not square-free and has a particular structure. The proof of the latter theorem does not use Theorem~\ref{L(n) for large n}.

\begin{lemma}\label{n_1n_2 not in L(n_2)}
For all $n_1,n_2\geq0$, if the word $n_1n_2$ is not a factor of $L(n_2)$, then $L(n_1n_2)=n_1L(n_2)$.
\end{lemma}
\begin{proof}
Since $n_1n_2$ never occurs in $L(n_2)$ and $L(n_2)$ is square-free, $n_1L(n_2)$ is square-free unless $n_1=n_2$ in which case its only square factor is the prefix $n_1n_1$. Thus, $L(n_1n_2)$ and $n_1L(n_2)$ are both infinite words beginning with $n_1n_2$ whose only square factors are contained in the prefix $n_1n_2$. So by the definition of $L$, $L(n_1n_2)\preccurlyeq n_1L(n_2)$. If $L(n_1n_2)\prec n_1L(n_2)$, then $L(n_1n_2)[1:]\prec L(n_2)$, which is a contradiction since these are both infinite square-free words beginning with $n_2$. Therefore, $L(n_1n_2)=n_1L(n_2)$.
\end{proof}

This immediately describes all words of the form $L(nn)$:

\begin{theorem}
For all $n\geq0$, $L(nn)=nL(n)$.
\end{theorem}
\begin{proof}
Since $L(n)$ is square-free, it does not contain $nn$. The result then follows from Lemma~\ref{n_1n_2 not in L(n_2)}.
\end{proof}

\begin{theorem}\label{L(n1n2)=n1L(n2)}
For all $n_1 \geq 3$ and $n_2 \geq 3$, we have $L(n_1n_2) = n_1L(n_2)$.
\end{theorem}

\begin{proof}
By Lemma~\ref{n_1n_2 not in L(n_2)}, it is sufficient to show that the word $n_1n_2$ never appears in $L(n_2)$.

By Theorem~\ref{L(n) for large n}, $L(n_2) = Y_{n_2} \rho(\alpha(L(\varepsilon)))$. Since $\rho(\alpha(L(\varepsilon)))$ is grounded, it cannot contain $n_1n_2$. Also, the first letter of $\rho(\alpha(L(\varepsilon)))$ is 0, so $n_1n_2$ cannot lie over the boundary. For the prefix we have
\[Y_{n_2}=n_2T(n_2)=\psi_2((n_2-2)P_0(n_2-2)),\]
hence it is enough to show that $n_1n_2$ does not occur in any $\ell$-chunk $\psi_2(\ell)$, nor over any chunk boundary.

Since $\max(\psi_2(0))=3$, $n_1n_2$ could only occur in a 0-chunk if $n_1n_2=33$, which is not a factor of $\psi_2(0)$.
For the case $\ell\geq1$, $\psi_2(\ell)=\psi_1((\ell+1)P_0(\ell+1))$. Hence, it is sufficient to show that $n_1n_2$ cannot occur in any $\psi_1(\ell)$ nor over the chunk boundary of any $\psi_1(0\ell)$ or $\psi_1(\ell0)$. Indeed, since $\psi_1(0)=202101$, and all other $\psi_1(\ell)$ are grounded, then $n_1n_2$ cannot occur in any $\psi_1(\ell)$. Also, $\psi_1(\ell)$ ends with a 1 for all $\ell\geq0$, so $n_1n_2$ cannot occur in $\psi_1(0\ell)$ or in $\psi_1(\ell0)$.

Finally, to show that $n_1n_2$ does not occur over any chunk boundary of $\psi_2$ recall that for $\ell\geq1$, $\psi_2(\ell)$ has suffix
\[P_1(\ell+2)=\psi_1(P_0(\ell+1))=\psi_1(R_{\ell+2}[:-2]),\]
which has suffix $\psi_1(1)$, which ends with a 1, so $n_1n_2$ cannot lie over a $\psi_2(\ell0)$ chunk boundary. Also, $\psi_2(0)$ ends with 2 so $n_1n_2$ cannot lie over a $\psi_2(0\ell)$ chunk boundary.

Therefore, $n_1n_2$ cannot occur anywhere in $n_1L(n_2)$, except as a prefix. And so $L(n_1n_2)=n_1L(n_2)$.
\end{proof}

Experiments suggest the following related result.

\begin{conjecture}
For all $n \geq 3$, we have $L(n 1) = n L(1)$ and $L(n 2) = n L(2)$.
\end{conjecture}

For example, it appears that $L(31) = 3L(1)$ and $L(32) = 3L(2)$. Since Theorem~\ref{rho is L-commuting} implies that $L(0n)=\rho(L(n-1))$ for all $n>0$, we have a proven or conjectural description of $L(w)$ for all 2-letter words $w$ except for $L(1n)$ when $n>1$ and $L(2n)$ when $n\not\in\{0,2\}$. However, it does appear that these words also have structures related to the ruler sequence and to the other words discussed in this paper.\\

The rest of this section deals with the case when $w=uv$ is not square-free and is a particular decomposition of $w$. The next lemma describes this decomposition.

\begin{lemma}\label{w=pqs}
Let $w$ be any nonempty finite word containing a square. Then there is a unique decomposition $w=psq$ such that $sq$ is the maximal square-free suffix of $w$, and $p[-1]s$ is the maximal square prefix of $p[-1]sq$.
\end{lemma}
\begin{proof}
Any single letter is square-free, so $w$ is guaranteed to have some square-free suffix. Each suffix of $w$ has a different length, so the maximal square-free suffix, $sq$ is unique. Since $w$ contains a square, $sq$ is a proper suffix of $w$, so $p$ is nonempty and unique.

If $p[-1]sq$ is square-free, then it would be a square-free suffix larger than $sq$ which is a contradiction. Any square in $p[-1]sq$ cannot be contained in $sq$ which is square-free, so $p[-1]sq$ has a square prefix. No two distinct prefixes of have the same length, so the maximal square prefix of $p[-1]sq$ is unique.
\end{proof}

\begin{remark*}
In the $w=pqs$ decomposition in Lemma~\ref{w=pqs}, $p$ and $s$ are always nonempty, while $q$ can be empty. The last letter of $w$ is always square-free, so $sq$ is always nonempty. Since $w$ contains a square, we cannot have $sq=w$ so $p$ is nonempty. Since $p[-1]$ is a single letter, it cannot be a square, so $s$ must be nonempty.
\end{remark*}

\begin{example*}
For $w=012323045$, we have that $p=012$, $s=323$, $q=045$.

For $w = 1121123210$, we have that
$p = 1121$,
$s = 1$, and
$q = 23210$.

For $w = 11011$, we have that
$p = 1101$,
$s = 1$, and
$q = \varepsilon$.
\end{example*}

\begin{theorem}\label{sufficient check}
Let $w$ be any nonempty finite word containing a square. Write $w=psq$ such that $sq$ is the maximal square-free suffix of $w$, and $p[-1]s$ is the maximal square prefix of $p[-1]sq$. Then $L(psq)=pL(sq)$ if and only if $L(psq)[:2|ps|]=(pL(sq))[:2|ps|]$.
\end{theorem}

In other words, to verify that $L(psq)=pL(sq)$, it is sufficient to verify that they match for their first $2|ps|$ letters. Note that this is potentially useful because $L(sq)$ is the maximal square-free tail of $pL(sq)$. Before proving Theorem~\ref{sufficient check}, we look at a few examples.
\begin{example*}
For $w=012323045$, since $ps=012323$  the theorem implies that $L(012323045)=012L(323045)$ if and only if their first $2|ps|=12$ letters match. Since $|w|=9$, and both words have $w$ as a prefix, it is sufficient to compute the next three letters of each. In this case, both have $010$ as their next letters and so we conclude that $L(012323045)=012L(323045)$.

For $w = 1121123210$, since $ps= 11211$ we obtain that $L(1121123210)=1121L(123210)$ are equal if and only if they match for the first $2|ps|=10$ letters. Since $|w|=10$, no further computations are necessary.

For $w = 11011$, since $ps=11011$ we have that
$L(11011)=1101L(1)$ if and only if they match for the first $2|ps|=10$ letters. However, in this case $L(11011)[:10]=1101120102\neq1101101201=1101L(1)[:10]$. 
\end{example*}

\begin{proof}[Proof of Theorem~\ref{sufficient check}]
The forward direction is trivial. We prove the other direction by induction. The base case is our supposition that $L(w)[:2|ps|] = (pL(sq))[:2|ps|]$. Now suppose $L(w)[:n] = (pL(sq))[:n]$ for some $n \geq 2|ps|$. We will prove that $L(w)[:n+1]=(pL(sq))[:n+1]$. We let $a = L(w)[n]$ and $b = (pL(sq))[n]=L(sq)[n-|p|]$, then we need to show that $a=b$.

Suppose that $a < b$, which implies that $(pL(sq))[:n]a$ has a square suffix. By the inductive hypothesis, $(pL(sq))[:n]a=L(w)[:n]a=L(w)[:n+1]$. Since $L(psq)$ and $pL(sq)$ both begin with $w$, we have that $n+1>|w|$, since otherwise $a=b$ is a letter of $w$. Also, since $L$ does not introduce new squares, $L(w)[:n+1]$ cannot have a square suffix. Therefore we cannot have $a<b$.

Now suppose that $a > b$, which analogously to the previous case implies that $L(w)[:n]b$ has a square suffix, say $yy$. By the inductive hypothesis, $L(w)[:n]b=(pL(sq))[:n]b=(pL(sq))[:n+1]$. Since $sq$ is square-free, $L(sq)$ is square-free, so the square must start in the prefix $p$. Let $k$ be the length of the suffix of $p$ contained in the square. Then $1\leq k\leq|p|$ and we have
\[\big(pL(sq)\big)[:n]b=p\big(L(sq)[:n-|p|]\big)b\]
so the square has length
\[|yy|=k+n-|p|+1\geq k+2|ps|-|p|+1=k+|p|+2|s|+1.\]
Note that the first half of the square cannot contain all of $s$, otherwise it would contain $p[-1]s$ which is a square. This in turn would imply that the second half contains $p[-1]s$ and is a factor of $L(sq)$ which is square-free.
Therefore, the first half of the square ends within $ps[:-1]$, and so $|y|<k+|s|$.

This implies that $|yy|=2|y|<2k+2|s|\leq k+|p|+2|s|$ which is a contradiction since $|yy|=k+|p|+2|s|+1$. Therefore, $a>b$ is neither possible.

This proves the inductive step, hence $L(w)[:n]=(pL(sq))[:n]$ for all $n$, as wanted.
\end{proof}

%%%%%%%%%%%%%%%%%%%%%%%%%%%%%%%%
\section{Inducing factors}\label{section: inducing factors}

In this section we consider the following problem.
Given a finite square-free word $w$, find
a word $p$ (not necessarily square-free) such that $p$ generates $pw$, i.e.\ $L(p)=L(pw)$.

\begin{definition*}
Let $w$ be a finite square-free word. For $0\leq j<|w|$ and $0 \leq k < w[j]$, we call a nonempty word of the form 
\[r_{j,k}(w) = w[:j]k\]
a \emph{restriction} of $w$. That is, $r_{j,k}(w)$ is obtained by starting with $w[:j+1]$ and decreasing the last letter by some amount. Let $m(w)$ be the total number of square-free restrictions of $w$, and relabel the square-free restrictions in the lexicographic order as $r_0(w),\dots,r_{m-1}(w)$, this is called the \emph{restriction sequence} of $w$.
\end{definition*}

\begin{example*}
For the word $w=2021$, we have 
\[r_{0,0}=0,\ r_{0,1}=1,\ r_{2,0}=200,r_{2,1}=201\textrm{, and }r_{3,0}=2020.\]
Taking only the square-free $r_{j,k}$ and sorting them lexicographically, we obtain that the restriction sequence of $w$ is  $r_0=0$, $r_1=1$, $r_2=201$.
\end{example*}

From the definition of lexicographic order, the set of restrictions is always totally ordered by $\prec$.
For the rest of this section, we will just write $r_{j,k}$, $r_i$, and $m$ where the dependence on $w$ is inferred by context.

Intuitively, the restriction sequence of $w$ does the following: Whenever we extend a word with $L$, we follow the ruler sequence until we can't anymore due to a square. So, to generate $w$ with $L$, we need to have exactly the right squares coming up at the right positions to deviate from the ruler sequence and spell out $w$ instead. The restriction sequence for $w$ is composed of all the words that will provide those necessary squares.

In other words, the restriction sequence is the collection of all square-free words that are not longer than $w$, and are lexicographically less than $w$. We then design $p$ so that $pr_i$ contains a square for all $i$. Thus, $p$ will generate the lexicographically least word that is greater than all $r_i$, which is $w$.

\begin{definition*}
Let $w$ be a finite square-free word, and let $v_i(w) = \max(w)+i+1$ and $V_i(w)=v_i\cdots v_1v_0$. For $i \leq m=m(w)$, we define words $x_i(w)$ by
\begin{align*}
    x_0 &= v_0, \\
    x_i &= v_i x_{i-1} r_{i-1} x_{i-1},
\end{align*}
where $v_i=v_i(w)$, $V_i=V_i(w)$, and $x_i=x_i(w)$. This dependence on $w$ will be inferred by context.

Note that for all $j \leq i$, $x_i$ has suffix $x_j$ and so $x_m$ has all $x_j$'s as suffixes. Also, for all $i$, $\max(r_i)\leq\max(w)<v_0<v_1<v_2<\cdots$ and $\max(x_i)=v_i$.
\end{definition*}

\begin{example*}
Continuing with the example $w=2021$, since $\max(w)=2$, we have $v_i=i+3$ for $0\leq i\leq m=3$. Then
\begin{align*}
    x_0&=v_0&=3,\\
    x_1&=v_1x_0r_0x_0&=4303,\\
    x_2&=v_2x_1r_1x_1&=5\ 4303\ 1\ 4303,\\
    x_3&=v_3x_2r_2x_2&=6\ 5430314303\ 201\ 5\ 4303\ 1\ 4303.
\end{align*}
We can write $x_3$ with spacing suggestive of the next lemma:
\begin{align*}
    x_3&=6543\cdot 03\cdot 143\cdot 03\cdot 201543\cdot 03\cdot 143\cdot 03\\
    &=V_3\cdot r_0V_0\cdot r_1V_1\cdot r_0V_0\cdot r_2V_2\cdot r_0V_0\cdot r_1V_1\cdot r_0V_0.
\end{align*}
\end{example*}

\begin{lemma}\label{structure of x_i}
Let $w$ be a finite square-free word, and let $\phi_w$ be the morphism defined on the alphabet $\{0,1,\dots,m\}$ by $\phi_w(k) = r_k V_k$ for letters $k$. Then for $0\leq i\leq m$, $x_i = V_i\phi_w(R_i[:-1])$.
\end{lemma}

\begin{proof}
We proceed by induction. For the base case,
\[
	x_0 = v_0 = V_0 = V_0 \phi_w(\varepsilon) = V_0 \phi_w(R_0[:-1]).
\]
For the inductive step, suppose $x_i = V_i \phi_w(R_i[:-1])$ for some $i<m$. Then
\begin{align*}
    x_{i+1} &= v_{i+1} x_i r_i x_i \\
    &= v_{i+1} V_i\phi_w(R_i[:-1])\ r_i\ V_i \phi_w(R_i[:-1]) \\
    &= V_{i+1} \phi_w(R_i[:-1])\ \phi_w(i)\ \phi_w(R_i[:-1]) \\
    &= V_{i+1} \phi_w(R_i[:-1]iR_i[:-1])\\
    &= V_{i+1} \phi_w(R_{i+1}[:-1]).
    \qedhere
\end{align*}
\end{proof}
Finally we present the main result of this section, which implies that for every finite square-free word $w$, there exists a prefix $p$ that generates $pw$. Indeed, it states that such prefix is given by $p=x_m$.

\begin{theorem}\label{x_m generates x_mw}
Let $w$ be a finite square-free word. Then $x_m$ generates $x_mw$.
\end{theorem}
\begin{proof}
We will show that for all $0\leq j<|w|$, $x_mw[:j]$ generates $x_mw[:j+1]$. This means that $L(x_mw[:j])=L(x_mw[:j+1])$, which proves the desired result:
\[L(x_m)=L(x_mw[:0])=L(x_mw[:1])=\dots=L(x_mw[:|w|])=L(x_mw).\]

To show that $L(x_mw[:j])=L(x_mw[:j+1])$ we need to prove the last letter of $x_mw[:j+1]$, which is $w[j]$, is irreducible and does not introduce a square.

First we prove the irreducibility condition. Let $0\leq j<|w|$ and $\ell<w[j]$, we need to show that ${x_mw[:j]\ell}$ has a square suffix. Indeed, since $0\leq j<|w|$ and $0\leq\ell<w[j]$, we have that $w[:j]\ell=r_{j,\ell}(w)$. If $w[:j]\ell$ contains a square, since $w$ is square-free, this must be a square suffix, and we are done. Otherwise, if $w[:j]\ell$ is square-free, then $w$ has a restriction $r_i=w[:j]\ell$ for some $i<m$. Hence $x_m$ has suffix $x_{i+1}$, which means that $x_m w[:j]\ell$ has suffix
\[x_{i+1}\ w[:j]\ell = v_{i+1}x_ir_ix_i\ r_i,\]
which also has a square suffix. Therefore $w[j]$ is irreducible in $x_mw[:j+1]$.

Now we prove that $w[j]$ does not introduce a square in $x_mw[:j]$. Suppose toward a contradiction that $x_m w[:j+1]$ has square suffix $yy$. Since $w$ is square-free, the square must start in the prefix $x_m$, and so $y$ contains the last letter of $x_m$ which is $x_0=v_0$. Hence, since $v_0>\max(w)$, $y$ cannot be completely contained in $w$. Therefore, the second half of the square starts in $x_m$ and contains all of $w[:j+1]$. This implies that $y$ has suffix $v_0w[:j+1]$.

By Lemma~\ref{structure of x_i},
\[x_m = V_m\phi_w(R_m[:-1])=V_m\cdot r_0V_0\cdot r_1V_1\cdot r_0V_0\cdot r_2V_2\cdots r_0V_0,\]
where $\phi_w(k)=r_kV_k$ for letters $k<|w|$. Since $y[-1]=w[j]\le \max(w)$ and all letters in any $V_k$ are greater than $\max(w)$, we have that the last letter of the first half of the square is in a factor $r_i$ for some $i$. Therefore, there are no partial $V_k$'s in the second half. 

Since $\max(\phi_w(k))=v_k>\max(w)$, the largest letter in $y$ is some letter $V_k[0]=v_k$ occurring $k$-chunk. Let $v_\ell=\max(y)$, then every occurrence of $v_\ell$ in the second half of the square is as the first letter of $V_\ell$ in an $\ell$-chunk. In order to contain $v_\ell$, the first half of the square must overlap a $k$-chunk  with $k\ge \ell$. From the structure of the ruler sequence, we know that the latest $k$-chunk with $k\ge \ell$ before the first $\ell$-chunk in the second half is an $(\ell+1)$-chunk. The first half cannot contain the letter $v_{\ell+1}$ in this $(\ell+1)$-chunk, but it must include the letter $v_\ell$. Therefore, $v_\ell$ is the first letter of the first half and since there are no partial $V_k$'s in the second half, then $y$ has prefix $V_\ell$. The second half then begins in the middle of an $\ell$-chunk at the beginning of the factor $V_\ell$ and so the first half has suffix $r_\ell$ from the same $\ell$-chunk. Before this factor $r_\ell$ there is another chunk, which always ends with $v_0$ and so $y$ has suffix $v_0r_\ell$. 

We have proved that the words $v_0w[:j+1]$ and $v_0r_\ell$ are both suffixes of $y$. Since $v_0>\max(w[:j+1])$ and $v_0>\max(r_\ell)$, we have that $r_\ell=w[:j+1]$. But this is a contradiction since the restriction $r_\ell$ cannot be a prefix of $w$. This proves that no such square exists and the result follows.
\end{proof}
\begin{example*}
Again using $w = 2021$, we can verify that
\[x_3 = 654303143032015430314303\textrm{ generates }x_3 w,\]
for example the suffix $303$ prevents a $0$, the suffix $430314303$ prevents a $1$, so since $2$ does not introduce a square, it can be located after this suffix. Similarly, we can continue checking that the whole word $w$ is the lexicographically least extension of $x_3$.
\end{example*}

\section{Glossary}\label{glossary}
A list of all the important mathematical objects in the paper, along with their definitions.

\subsection{Sequences of Words}

$R_n$ is defined for all letters $n\geq0$, and is the ruler sequence up to the first appearance of $n$. We can also define $R_n$ by $R_n = \rho^n(0)$, or inductively by $R_0 = 0$ and $R_n = R_{n-1} R_{n-1}^+$. $R_n$ is always even-grounded, the length of $R_n$ is $2^n$, and $\max(R_n) = n$.

$b_n$ is defined for letters $n\geq2$, and is defined inductively by $b_2 = 0102012021012$ and $b_n = b_{n-1} b_{n-1}^+ R_n$. $b_n$ is never grounded, the length of $b_n$ is $2^{n-2}(4n+5)$, and $\max(b_n) = n$.

$c_n$ is defined for letters $n\geq3$, and is defined inductively by $c_n = c_{n-1} c_{n-1}^+ R_n b_n$ with a base case $c_3$ of length 261. $c_n$ is never grounded, the length of $c_n$ is $2^{n-3}(4n^2+22n+159)$, and $\max(c_n) = n$.

\subsection{Functions on Letters}

$P_0(n)$ is defined for all letters $n\geq0$, and is the largest prefix of the ruler sequence such that $n P_0(n)$ is a prefix of $L(n)$.
The length of $P_0(n)$ is $2^{n+1}-2$, so $P_0(n)=R_{n+1}[:-2]$.

$P_1(n)$ is defined for letters $n\geq3$, and is equal to $\psi_1(P_0(n-1))$. The length of $P_1(n)$ is $(4n+1)2^{n-1}-5$.

$P_2(n)$ is defined for letters $n\geq3$, and is equal to $\psi_2(P_0(n-2))$. The length of $P_2(n)$ is $(4n^2+14n+149)2^{n-2}-193$.

$T(n)$ is defined for letters $n\geq3$, and is equal to $P_0(n) P_1(n) P_2(n)$. The length of $T(n)$ is $(4n^2+22n+159)2^{n-2}-200$.

\subsection{Morphisms}

The ruler morphism $\rho$ is defined by $\rho(n) = 0(n+1)$.

$\psi_1$ is defined by $\psi_1(0) = 202101$ and $\psi_1(n) = (n+1) P_0(n+1)$ for $n > 0$. The length of $\psi_1(n)$ is $2^{n+2}-1$ for $n>0$.

$\psi_2$ is defined by $\psi_2(n) = (n+2) \, P_0(n+2) \, P_1(n+2)$ for $n > 0$, with $\psi_2(0)$ a specific word of length 199. For $n>0$, the length of $\psi_2(n)$ is $(4n+13)2^{n+1}-6$.

$\alpha$ is defined by
\[\alpha(n) = \begin{cases} 
    EFE & \text{if $n=0$} \\
	B_1 \: R_4 \: C \: B_1 \: R_4 & \text{if $n=1$} \\
	\alpha(n-1)^+ \: R_{n+3} \: C \: \alpha(n-1)^+\: R_{n+3} & \text{if $n \geq 2$}
\end{cases}
\]
for constants $C$, $B_0$, $B_1$, $E$, and $F$.

\subsection{Constants}
$\varepsilon$ is the empty word.

$C = 0102030102$ is a grounded word of length 10 which is used in the inductive definition of $\alpha$. While not mathematically relevant, we would be remiss not to note that coding $C$ into letters yields the word $\textsf{abacadabac}$, which is amusingly similar to $\textsf{abracadabra}$.

$B_0 = 0301 \; \psi_1(1010)[:-3] \; \psi_2(1010)[:-6] \; \psi_2(10)[:-12] \; 301020$ is a non-grounded word of length 798 which is used in the definition of $\alpha(0)$.

$B_1 = \rho(B_0[7:-5])$ is a grounded word of length 1572 which is used in the definition of $\alpha(1)$, which is the base case for the inductive definition of $\alpha$.

$E=0102B_01B_0[:-9]$ and
$F=B_0[-9:] 3010302 C 0103 C^+ 02$. These are useful in the structure of $\alpha$ because $EFE=\alpha(0)$.

$G = 010203012$ is a word of length 9, which is the shortest prefix that generates $\alpha(0)$ and appears in various proofs.

$A$ is a word of length 13747 with the property that $n T(n) A$ is a prefix of $L(n)$ for all $n \ge 3$.

%%%%%%%%%%%%%%%%%%%%%%%%%%%%%%%%
\section*{Acknowledgements}
We thank the organizers of the Polymath Jr.\ program 2021 for creating the environment that allowed this collaboration.
We also thank our colleagues Alycia Doucette, Bridget Duah, Bill Feng, Mordechai Goldberger, Luke Hammer, Ziqi He, Amanda Lamphere, Mary Olivia Liebig, Jacob Micheletti, Adil Oryspayev, Sara Salazar, Shiyao Shen, Wangsheng Song, and Thomas Sottosanti, who contributed in coding, generating data, and presenting results.

%%%%%%%%%%%%%%%%%%%%%%%%%%%%%%%%

\end{document}